\documentclass[11pt]{article}
\usepackage[latin1]{inputenc}
\usepackage[english]{babel}
\usepackage{amssymb}
\usepackage{amsmath,amsthm}
\usepackage[normalem]{ulem}

\newtheorem{defn}{Definition}[section]
\newtheorem{thm}{Theorem}[section]

\newtheorem{cor}{Corollary}[section]
\newtheorem{lem}{Lemma}[section]
\newtheorem{rmk}{Remark}[section]
\newtheorem{exmp}{Example}[section]

\usepackage{accents}
\newcommand{\dbtilde}[1]{\accentset{\approx}{#1}}

\makeatletter

\let \al=\alpha
\let \be=\beta
\let \var=\varphi
\let \vare=\varepsilon

\let \de=\delta

\let \la=\lambda

\let \ga=\gamma

\let \q=\quad
\let \qq=\qquad
\let \med=\medskip
\let \smal=\smallskip
\let \dps=\displaystyle

\DeclareMathOperator{\e}{e}
\DeclareMathOperator{\diag}{diag}

\newcommand{\R}{\mathbb{R}}
\newcommand{\N}{\mathbb{N}}

\def\system#1{\left\{\null\,\vcenter{
\ialign{\strut\hfil$##$&$##$\hfil&&\enspace$##$\enspace&
\hfil$##$&$##$\hfil\crcr#1\crcr}}\right.}

\begin{document}

\vspace{0.5cm}
 {\Large \centerline{\bf Stability for nonautonomous
 linear  differential systems with infinite delay}}

 \

 \centerline{\scshape Teresa Faria}
 
 \smal
 \centerline{
 Departamento de Matem\'atica and CMAF-CIO, Faculdade de Ci\^encias, Universidade de Lisboa}
  \centerline{
Campo Grande, 1749-016 Lisboa, Portugal}
 
\centerline{  Email:
teresa.faria@fc.ul.pt}

\vskip .5cm

\centerline{\it To the memory of Professor  Russell A. Johnson}

\

\begin{abstract}  We study the stability of general $n$-dimensional nonautonomous linear differential equations  with  infinite  delays.  Delay independent  criteria, as well as criteria depending on the size of some finite delays are established.  In the first situation,  the effect of the delays is dominated   by non-delayed diagonal negative feedback terms, and sufficient conditions for both
the   asymptotic   and the exponential asymptotic stability of the system are given. In the second case, the  stability depends on the size of some bounded diagonal delays and coefficients, although terms with unbounded delay may co-exist.  Our results encompass  DDEs with discrete and distributed delays, and enhance some recent achievements in the  literature.
\end{abstract}



\section{Introduction}
\setcounter{equation}{0}

In this paper,
the  focus is to investigate the asymptotic and exponential stabilities of  a general nonautonomous linear system of delay differential equations (DDEs) with infinite delay, 
\begin{equation}\label{Lin0}
x'(t)={\cal L}(t)x_t,\q t\in I,
\end{equation} where  $I=[t_0,\infty)$ for some $t_0\in \R$, ${\cal L}(t)$ is in $L({\cal C},\R^n)$, the usual space of bounded linear operators from ${\cal C}$ to $\R^n$ equipped with  the operator norm, 
 and ${\cal C}$ is an adequate Banach space of continuous functions defined on $(-\infty,0]$ with values in $\R^n$.  
 As usual, $x_t$ denotes the entire past history of the system up to time $t$, or, in other words, $x_t(s)=x(t+s)$ for $s\le 0$.   For simplicity, here one  assumes that $ (t,\phi)\mapsto {\cal L}(t)\phi$ is continuous,  although  one could consider the more general framework of  $t \mapsto {\cal L}(t)\phi$ a Borel measurable function for each $\phi$, with $\|{\cal L}(t)\|$   bounded on $I$ by a  function $m(t)$ in $L_{\rm loc}^1(I;\R)$. 
 See \cite[Chapter 4]{HMN} for more details. 
 
The stability of  autonomous and nonautonomous linear DDEs has been the subject of intensive studies. Even for ordinary differential equations (ODEs), the nonautonomous situation is not easy to address in its generality, see e.g. important contributions by  Coppel \cite{Coppel}, 
Johnson and Sell \cite{JohnsonSell}, Sacker and Sell \cite{sase}. On the other hand, the introduction of large delays in differential equations may lead to oscillations,  loss of stability
of equilibria,  and existence of unbounded solutions.  For linear DDEs, delay independent  results for stability as well as criteria depending on the size of the delays have been established, for both scalar and multi-dimensional equations. The literature on this subject is very vast: here we only refer to some monographs \cite{Gop,HaleLunel,Kuang} and  a few selected papers \cite{AGP,Driver,FW,HK,HofbauerSo, Krisztin,SoYuChen, STZ,Yoneyama}.
%
%
%
 
 For the last few years, there has been a renewed interest in  the    analysis of stability of nonautonomous linear 
 DDEs,  and several methods and tools have been proposed, see e.g.~\cite{bb11,BDSS18,BDSS19,FeMartonPituk,GPS,GH,GH18, Hatvani16,NgocCao16,Ngoc19} and references therein.
 The main goal of this paper is to obtain new explicit sufficient conditions for the asymptotic and exponential asymptotic stability of a general linear system \eqref{Lin0}, which improve  and generalize some criteria
 in recent literature. We emphasize that here we consider very general linear DDEs with possible unbounded delays, both discrete and distributed, while typically most  authors  impose   the  delays  to be  finite or discrete, or both. Moreover,  the a priori boundedness of all the coefficients will not be required.  Two types of criteria will be obtained, depending on  whether  system  \eqref{Lin0}  possesses   diagonal  terms without delay which dominate the effect of the delayed terms, or not. The latter case is  not often treated in the literature, although there have been some recent interesting developments in this area \cite{bb11,BDSS18,BDSS19,FeMartonPituk,STZ}, following different approaches: $3/2$-stability conditions, Lyapounov  functionals, theory of monotone systems, asymptotic equivalence to  linear ODEs, etc.




The method employed here is based on an  auxiliary simple result,  which states that, 
under some algebraic conditions and without imposing the boundedness of coefficients and delays, 
the norm along  solutions is nonincreasing.  Special care is however required to deal with the infinite delay.
 Our techniques are very different from others proposed in the literature, though
 the results presented in this paper were  inspired by  some previous works, 
 which prompted  us to search for either sharper  or more embracing criteria. Some of our concrete purposes are described below.
 
  In  \cite{FO08}, Faria and Oliveira gave sharp conditions for  the exponential asymptotic stability of {\it autonomous} linear systems with finite delay and dominating instantaneous negative feedbacks.  The analysis in  \cite{FO08} was further pursued in \cite{Faria10}, for the case of infinite delay. Ngoc and Cao \cite {NgocCao16} considered a linear system, again with dominating  diagonal  terms without delay, of the form $x'(t)=-D(t)x(t)+{\cal L}(t)x_t$, where ${\cal L}(t)$ has the form ${\cal L}(t)\phi=\int_{-\infty}^0 B(t,s)\phi(s)\, ds$ and $D(t), B(t,\cdot)$ are $n\times n$ matrices of continuous functions, but assumed that  $D(t),B(t,\cdot)$ are bounded by some autonomous matrices. One of the goals of the present paper is to remove this constraint. Recently, Hatvani \cite{Hatvani16} and Gy\"ori and Horv\'ath \cite{GH}   achieved sharper results for the asymptotic stability of {\it scalar} differential equations $x'(t)= -d(t)x(t)+\be (t)x(t-\tau(t))$ and inequalities $x'(t)\le -d(t)x(t)+\be (t)x(t-\tau(t))$,  respectively, without the a priori requirement of having bounded coefficients. We shall show that some of the results in \cite{GH,GH18,Hatvani16,NgocCao16} are a simple consequence of the stability criteria established here for $n$-dimensional linear DDEs. On the other hand,  there are several recent works   where  explicit conditions for  the exponential asymptotic stability of  linear  DDEs  depending on the size of delays were found,
  see Berezansky and Braverman \cite{bb11,bb20} for the scalar case and  Berezansky et al. \cite{BDSS18,BDSS19} for $n$-dimensional systems, as well as references therein. However, in  \cite{bb11,BDSS18}  only the situation of time-varying {\it bounded} and {\it discrete} delays was considered -- two constraints removed in this work. In any case, for criteria depending on the delays, clearly  constraints on the size of  some  diagonal delays must be imposed.

We observe that not only the asymptotic stability of general  linear equations \eqref{Lin0} is important {\it per se}, but also that it has relevant consequences in the study of the global dynamics of nonautonomous  DDEs 
\begin{equation}\label{NonLin}
x'(t)={\cal L}(t)x_t+f(t,x_t),
\end{equation} 
where $f$ is smooth on some open subset of $\R\times {\cal C}$. This system can be seen as a perturbation of \eqref{Lin0},  and the stability or instability of its linearization at e.g. 0 (if $f(t,0)=0, D_2f(t,0)=0$ for all $t$) is a key ingredient to  further analyze the large-time behavior of solutions, in terms of local or  global asymptotic stability,  existence of oscillatory solutions, and many other features. In the autonomous case $x'(t)={\cal L}x_t+f(x_t)$, where   $f(0)=0,f'(0)=0$,  the well-known  {\it principle of linearized stability} is valid for equations with infinite delay \cite{DG}. The nonautonomous situation is certainly more difficult to analyze, but   a crucial idea is to use the variation of constant formula and the stability properties of the linearized system, possibly coupled with additional properties of the nonlinear perturbation $f$, such as monotonicity or boundedness, to further derive sufficient conditions for the stability, persistence and permanence of \eqref{NonLin}. This methodology was used for instance in \cite{FOS}, where  the authors studied the asymptotic behavior of solutions for a family of  nonlinear DDEs obtained as  perturbations of an ODE, given by $x'(t)=A(t)x(t)+f(t,x_t)$, with $A(t)$ an $n\times n$ matrix of continuous functions such that the ODE $x'(t)=A(t)x(t)$ is exponentially stable, and $f$   of the form $f(t,\phi)=(f_1(t,\phi_1),\dots, f_n(t,\phi_n))$
for $t\ge 0$ and  $\phi=(\phi_1,\dots,\phi_n)$ with $(t,\phi)\in {\rm dom}\, f$.
In fact, a main motivation for this work was to first  address the stability of linear DDEs \eqref{Lin0}, in order to extend the results in  \cite{FOS} to some classes of nonautonomous linear DDEs \eqref{NonLin} with  infinite delay.

 \smal

This paper is organized as follows. In Section 2, a suitable phase space ${\cal C}$ to treat DDEs with infinite delay is chosen, and some notation introduced. In Section 3, we   consider a linear DDE with  dominant nondelayed terms, start with some auxiliary results and then  establish delay-independent sufficient conditions for both its   exponential and asymptotic stabilities. In the latter case, some further restrictions on the general form of \eqref{Lin0} are imposed. 
In Section 4, we study the stability of  \eqref{Lin0} without assuming  the existence and dominance of diagonal instantaneous negative feedback terms; nevertheless  it turns out that the size of the diagonal  coefficients  and finite delays       will be decisive to derive our stability criteria, although unbounded delays may co-exist. Through Sections 3 and 4, we compare our results with some achievements in recent literature. In the last section, some illustrative examples are presented.


\section{Phase space and notation}
\setcounter{equation}{0}

In this preliminary section, we recall an abstract framework to deal with  DDEs with infinite delay.  In view of the  unbounded delays, the phase space ${\cal C}$ should satisfy some fundamental axioms which guarantee that it is `admissible', so that the classical results of existence, uniqueness, continuation for the future, and continuous dependence  of solutions on the initial data are valid -- a subject well establish in the literature. A convenient choice of ${\cal C}$ is set below, however other spaces are possible.


Consider a weight function $g$  satisfying the following properties:
  \begin{itemize}
\item[(g)] $g:(-\infty ,0]\to [1,\infty)$ is a nonincreasing continuous function such that  $g(0)=1$,  $\dps\lim_{s\to -\infty}g(s)=\infty$ and ${\lim_{u\to 0^-}{{g(s+u)}\over {g(s)}}=1}$ uniformly on $(-\infty ,0]$.
\end{itemize}

For each $n\in\N$, define the Banach space
$$C_g^0=C_g^0(\R^n):=\left\{ \phi\in C((-\infty, 0];\R^{n}) :\lim_{t\to -\infty}\frac{|\phi(s)|} {g(s)}=0\right\},$$
 with the norm
$$\|\phi\|_g=\sup_{s\le 0}{{|\phi(s)|}\over {g(s)}},$$
and $|\cdot|$  any chosen norm  in $\R^n$. 
This space is an {\it admissible} Banach phase space in the sense that it satisfies  the required axioms 
(A), (B) and (C2) of \cite{HMN}. For instance, for any $\ga >0$ the function $g(s)=\e^{-\ga s},\, s\le 0,$   satisfies the properties in (g); for such $g$, the notations $C_\ga^0:=C_{\e^{-\ga \cdot}}^0$ and $\|\phi\|_\ga:=\sup_{s\le 0} \e^{\ga s}|\phi(s)|$ are used. Alternatively, the space 
$
C_\ga=C_\ga(\R^n):=\big\{ \phi\in C((-\infty, 0];\R^{n}) : \lim_{t\to -\infty} \e^{\ga s}|\phi(s)|\ {\rm exists}\big\}
$
 with the same norm $\|\phi\|_\ga$ is   often considered in the literature.

In ${\cal C}=C_g^0$, an $n$-dimensional DDE with infinite delay is
  written in abstract form    as
\begin{equation}\label{2.1}
x'(t)=f(t,x_t),
\end{equation}
where $f:D\subset \R\times {\cal C}\to \R^n$ is continuous (or satisfies the Caratheodory conditions), and, 
as usual, the entire past of unkown solutions in the phase space ${\cal C}$ are denoted by $x_t$: $x_t(s)=x(t+s), s\le 0$. 

The space $C_g^0$ as well as $C_\ga$ are  always   {\it fading memory} spaces, which provides some further important properties for solutions of \eqref{2.1} \cite{MurakamiNaito}. In what concerns linear {\it autonomous} equations $x'(t)=Lx_t$, with $L\in L({\cal C},\R^n)$, it is well known that,  if ${\cal C}$ is a fading memory space, then
the zero solution  is asymptotically stable if and only if all the roots of the characteristic equation have negative real parts.  See \cite{HMN,MurakamiNaito} for definitions, results  and more properties. 


 Clearly, the case of systems  $x'(t)=f(t,x_t)$ with finite delay is included in
the present setting. In fact, for DDEs with finite delay $\tau\ge 0$,
take a weight function $g:(-\infty ,0]\to [1,\infty)$ such that $g(s)\equiv 1$ on $[-\tau,0]$ and
(g) holds. Thus, the space $C([-\tau, 0];\R^n)$  with the usual sup norm $\|\cdot\|_\infty$ can
be seen as a closed subspace of $C_g^0$ with the  norm $\|\cdot\|_g$.

The dual $(C_g^0(\R))'$ of $C_g^0(\R)$ is identified with the space $M_g((-\infty,0];\R)$ of Borel measures  $\mu:(-\infty,0]\to \R$, in the sense that each bounded linear functional $T:C_g^0(\R)\to\R$
 is    represented by a real Borel measure  $\mu: (-\infty,0]\to\R$,
\begin{equation*}
 T(\psi)=\int_{-\infty}^0\psi(s)\, d\mu(s),\q t\ge 0, \psi\in C_g(\R), 
   \end{equation*}
  with operator norm 
  $\|T\|=Var_{(-\infty,0]}(g\mu):=\int_{-\infty}^0g(s)\, d|\mu|(s)<\infty$
\cite{Rudin}. Thus, an operator $L\in L({\cal C},\R^n)$ is identified with an element $\eta=[\eta_{ij}]_{n\times n}$ in the space $M_g((-\infty,0];\R^{n\times n})$  of $n\times n$ matrix-valued Borel measures  on $(-\infty,0]$,    $\eta_{ij}\in M_g((-\infty,0];\R)$, in such a way that
$L\phi=\int_{-\infty}^0 [d \eta(s)]\phi(s),$ 
with norm    
$\|L\|=\|\eta\|_g:=\int_{-\infty}^0g(s)d|\eta|(s)<\infty$, where $|\eta|(s)$ is the total variation measure of $\eta(s)$.


\med
We now set some notation and terminology. As mentioned, the phase space  is a priori fixed as  ${\cal C}=C_g^0$, for some weight function $g$ satisfying (g). A vector $c$  in $\R^n$ is said to be {\it positive}   if all its components are positive, and we write $c>0$. Analogously, we define {\it nonnegative} vectors $c$, and {\it positive} and {\it nonnegative} functions $\phi\in {\cal C}$, with notation $c\ge 0, \phi>0, \phi\ge 0$, respectively.
A vector $c$ in $\R^n$  is identified in  ${\cal C}$  with the constant function $\phi(s)=c$ for $s\le 0$. For $c\in \R^n$,  it is understood that $c_i$  means the $i$th-component of $c$, for $1\le i\le n$. Analogously, $f_i$ is the $i$th-component of  a function $f$ with values in $\R^n$.

Unless otherwise stated, we suppose that $\R^n$ is equipped with the supremum norm, $|x|=|x|_\infty=\max_{1\le i\le n}|x_i|$, for $x=(x_1,\dots,x_n)\in\R^n$. If there is no possibility of misinterpretation, the norm $\|\cdot\|_g$ in ${\cal C}$ will be simply denoted by $\|\cdot\|$.
 For a positive vector $v=(v_1,\dots,v_n)$ we denote by $v^{-1}$ the vector  $v^{-1}=(v_1^{-1},\dots,v_n^{-1})$; we shall also consider  norms $|\cdot|_v$  defined by $|x|_v=\max_{1\le i\le n}(v_i|x_i|)$ for $x=(x_1,\dots,x_n)\in\R^n$ and   the corresponding norms in ${\cal C}$, given by $\|\var\|_v=\|\var\|_{g,v}=\sup_{s\le 0}g(s)^{-1}|\var(s)|_v$. Hereafter, we use  ${\bf 1}=(1,\dots, 1).$

Solutions of linear systems \eqref{Lin0} with initial conditions $x_\sigma=\phi$ $ (\sigma\in I, \phi\in {\cal C})$, i.e., $x(\sigma+s)=\phi(s),\, s\le 0,$
 are defined for all $t\ge \sigma$ \cite{HMN}; they are  denoted by  $x(t,\sigma,\phi)$ in $\R^n$, or  $x_t(\sigma,\phi)$ in ${\cal C}$. In what follows, let $I=\R^+:=[0,\infty)$, but any other choice of $I=[t_0,\infty)$  is possible.  Here, for simplicity, the concepts of asymptotic and exponential stability always refer to  stability on {\it some interval} $[\al,\infty)\subset \R^+$. Note also that, for linear systems, the asymptotic stability is equivalent to the stability and global attractivity of $x=0$. To be more precise, we will use the definitions below.

\begin{defn}\label{defn2.1}
The linear system  \eqref{Lin0}  is said to be {\bf stable} on $[\al,\infty)\subset \R^+$ if for any $\vare>0$ and $\sigma\ge \al$  there is $\de=\de(\vare,\sigma)>0$ such that  $\|x_t(\sigma,\phi)\|< \vare$ for all $t\ge \sigma$, whenever $\|\phi\|< \de$. System \eqref{Lin0} is {\bf   asymptotically stable}  if
 $x=0$ is   asymptotically stable  on some interval $[\al,\infty)\subset \R^+$; in other words, \eqref{Lin0} is  stable on  $[\al,\infty)$ and  
$\lim_{t\to\infty}x(t,\sigma,\phi)=0$ for all $\sigma\ge 0,\phi\in {\cal C}$;
 $x=0$ is  {\bf exponentially asymptotically stable}  if all solutions of \eqref{Lin0} tend to zero exponentially at infinity, uniformly on some interval $[\al,\infty)\subset \R^+$; i.e.,  there exist $\al, K,\be>0$ such that $\|x_t(\sigma,\phi)\|\le K\e^{-\be(t-\sigma)}\|\phi\|$ for all $t\ge \sigma\ge \al$ and $\phi\in {\cal C}$.   \end{defn}

%

 \section{Linear systems with  instantaneous diagonal dominance}
\setcounter{equation}{0}

In this section, we deal with linear DDEs for which  the effect of the delays is dominated  by non-delayed diagonal negative feedback terms. In order to  analyze the {\it absolute} stability, i.e., to set stability conditions which do not  dependent on the delays,
we separate nondelayed from delayed terms in ${\cal L}(t)$, so that ${\cal L}(t)$ has the form ${\cal L}(t)\phi=-D(t)\phi (0)+L(t)\phi$, and write
the  nonautonomous linear system \eqref{Lin0}  as
\begin{equation}\label{Lin3}
x'(t)=-D(t)x(t)+L(t)x_t,\q t\in I,
\end{equation} 
%
where $D(t)=[d_{ij}(t)]$ is an $n\times n$ matrix of functions on $I=\R^+$,  and $L:I\to L({\cal C},\R^n)$. 
Terms with  time-dependent discrete delays as well as distributed delays are all incorporated in $L(t)x_t$. 


For $L(t)$ as in \eqref{Lin3},  let $\eta(t)=[\eta_{ij}(t,\cdot)]_{n\times n}\in M_g((-\infty,0];\R^{n\times n})$ be such that
$$L(t)\phi=\int_{-\infty}^0 [d_s \eta(t,s)]\phi(s)\q {\rm for}\q \phi\in {\cal C},$$
with norm    given by
$\|L(t)\|=\|\eta(t)\|_g:=\int_{-\infty}^0g(s)d|\eta|(t,\cdot)(s)<\infty$.
Consider the components 
 $L(t)=(L_1(t),\dots,L_n(t))$, and write $L_i(t)\phi=\sum_{j=1}^nL_{ij}(t)\phi_j$ for $  t\ge 0, \phi=(\phi_1,\dots,\phi_n)\in {\cal C}$,
where each linear functional $L_{ij}(t):C_g(\R)\to \R$ is identified with $\eta_{ij}(t,\cdot)\in M_g((-\infty,0];\R)$.


   With \eqref{Lin0} written in  this form,   one may suppose that the operators $L_{ij}(t)$  are {\it non-atomic at zero}, i.e., $\eta_{ij}(t,0^-)=\eta_{ij}(t,0)$  for $t\ge 0$  (cf. \cite[Chapter 6]{HaleLunel}), this restriction however will not have any influence in the writing of our results.

Define 
the $n\times n$ matrix-valued functions 
 \begin{equation}\label{3.7}
 \begin{split}
 D(t)&=[d_{ij}(t)],\q \widehat D(t)=[\hat d_{ij}(t)]\ {\rm where}\ \hat d_{ij}(t)=\system{d_{ii}(t)\q &{\rm if}\ i=j\cr
 -|d_{ij}(t)|\q  &{\rm if}\ i\ne j\cr},\\
 A(t)&=\Big [\|L_{ij}(t)\|\Big ]\q {\rm and}\q M(t)=\widehat D(t)-A(t), \q {\rm for} \q t\in [0,\infty).
 \end{split}
 \end{equation}

For \eqref{Lin3},  in the sequel  we assume the general hypotheses:
  \begin{itemize}
\item[(H1)] the functions $d_{ij}:[0,\infty)\to \R, L_{ij}:[0,\infty)\to L(C_g^0(\R),\R)$ are continuous,  for $i,j=1,\dots,n$;
\item[(H2)]   there exist a vector $v> 0$ and $T\ge 0$ such that $M(t)v\ge 0$ for  $t\ge T$.
\end{itemize}

\begin{rmk}\label{rmk3.0}
{\rm To simplify the exposition, here  the regularity in (H1) is imposed. In fact, as mentioned in the Introduction, 
instead of continuous functions
one could consider the more general framework of $d_{ij}(t)$  in $L_{\rm loc}^1(\R^+;\R)$ and $t \mapsto L(t)\phi$  Borel measurable for each $\phi$, with $\|L(t)\|$   bounded on $\R^+$ by a  function $m(t)$ in $L_{\rm loc}^1(\R^+;\R)$. We stress that the proofs of our results do not depend on the continuity of the coefficients, as the reader can easily verify. Thus, it is important to notice that in particular they are  generalized in a straightforward way to linear DDEs with impulses, where coefficients and delays are piecewise continuous functions.}\end{rmk}

%

A dominance of the diagonal  instantaneous terms in \eqref{Lin3} is expressed by  condition (H2). 
 Note also that (H2) implies  that $d_i(t):=d_{ii}(t)\ge 0$ for all $t\ge 0,\, i\in\{ 1,\dots,n\},$  with the equality $d_i(t)=0$ if and  only if all the coefficients of the $i$th-lines of $D(t)$ and $A(t)$ are zero.  
   
  \begin{rmk}\label{rmk3.1} {\rm Consider \eqref{Lin3} under the general assumption (H1), and
suppose that  (H2) is satis\-fied. Set $a_{ij}(t):=\|L_{ij}(t)\|$. Rescaling the variables by $\bar x_i(t)=v_i^{-1}x_i(t)\, (1\le i\le n)$, where $v=(v_1,\dots,v_n)> 0$ is a vector as in (H2), we obtain a new linear DDE $\bar x'(t)=-\bar D(t)\bar x(t)+\bar L(t)\bar x_t$, where the corresponding matrices $\bar D(t)=[\bar d_{ij}(t)]$ and  $\bar A(t)=[\bar a_{ij}(t)]$ have entries $\bar d_{ij}(t)=v_i^{-1}d_{ij}(t)v_j$ and  $\bar a_{ij}(t)=v_i^{-1}a_{ij}(t)v_j$. In this way, and after dropping the bars for simplicity, we may consider a system \eqref{Lin3} for which  (H2) is valid with $v={\bf 1}:=(1,\dots,1)$. This scaling of $\R^n$ also transforms the norm $|x|_{v^{-1}}=\max_{1\le i\le n} v_i^{-1}|x_i|$  into the norm $|x|=\max_{1\le i\le n} |x_i|$. Throughout this paper, without loss of generality and whenever it is convenient,   if condition (H2) (or a stronger version of (H2), see (H4), (H5) below)  is satisfied,  we shall assume that it holds with the vector  $v={\bf 1}$.}
\end{rmk}

%
%
%

We start with some preliminary lemmas. 
The  auxiliary result below, although elementary,  plays a crucial role in our stability criteria.

\begin{lem}\label{lem3.1} If  assumptions (H1),  (H2) are satisfied,  the solutions of \eqref{Lin0} satisfy $|x(t,t_0,\phi)|_{v^{-1}}\le \|\phi\|_{g,v^{-1}}$ for $ t\ge t_0\ge T,\phi\in {\cal C}$, where $T,v$ are as in (H2). In particular, \eqref{Lin0}  is (uniformly) stable on $[T,\infty)$. \end{lem}

\begin{proof}  As described above, after  rescaling the variables by $\bar x_i(t)=v_i^{-1}x_i(t)\, (1\le i\le n)$ we may assume  (H2)  with $v={\bf 1}:=(1,\dots,1)$,  which in turn implies $d_i(t)\ge \sum_{j\ne i} |d_{ij}(t)|+\sum_j \|L_{ij}(t)\|$ for all $i$ and $t\ge T$,  and
$|x|_{v^{-1}}=\max_{1\le i\le n} |x_i|$.

Fix $\phi\in {\cal C}, t_0\ge T$, and consider the solution $x(t)=x(t,t_0,\phi)$ of \eqref{Lin0}. We claim that $|x(t)|\le \|x_{t_0}\|_g$ for $t$ on each interval $[t_0,t_0+a]\ (a>0)$.

Define $J=[t_0,t_0+a]$. For the sake of contradiction,  suppose that $\max_{t\in J}|x(t)|>\|x_{t_0}\|_g$.
Denote $u_j=\max_{t\in J} |x_j(t)|$ for $1\le j\le n$, and take $i\in \{1,\dots,n\}$ and $ t_1\in (t_0,t_0+a]$ such that $u_i=\max_j u_j=|x_i(t_1)|=|x(t_1)|$. 
For $s\le 0$ and $t\in J$, we get
$$\frac{|x(t+s)|}{g(s)}\le \frac{|x(t+s)|}{g(t-t_0+s)}\le \|x_{t_0}\|_g<|x_i(t_1)|\q {\rm if}\q s+t\le t_0$$
and
$$ \frac{|x(t+s)|}{g(s)}\le |x(t+s)|\le |x_i(t_1)|\q {\rm if}\q t_0\le s+t,$$
 hence
$ \|x_{t}\|_g\le u_i$ for $t\in J.$
 Now, we suppose that $x_i(t_1)>0$; the case $x_i(t_1)<0$ is analogous. 
 For $d_i(t)=d_{ii}(t)$, from (H2) we obtain
 $$x_i'(t)+d_i(t)x_i(t)\le \sum_{j\ne i} |d_{ij}(t)||x_j(t)|+\sum_j \|L_{ij}(t)\| \|x_{j,t}\|_g\le d_i(t)u_i,\q t\in J,$$
 thus
 $$x_i(t)\le x_i(t_0)\e^{-\int_{t_0}^t d_i(s)\, ds}+u_i(1-\e^{-\int_{t_0}^t d_i(s)\, ds}),\q t\in J.$$
 For $t=t_1$, we derive $x_i(t_0)-u_i\ge 0$, which contradicts the assumption $u_i>\|x_{t_0}\|_g$.
\end{proof}

\begin{lem}\label{lem3.2} Consider \eqref{Lin3}  and assume  (H1).
In addition, suppose that there exists a  measurable, locally integrable function $e:\R\to \R^+$,  such that  the following conditions are satisfied:

(i) the operators 
 \begin{equation}\label{tildeL} \tilde L_{ij}(t) (\psi):=L_{ij}(t) (\psi^{t,e}),
 \end{equation} where $\psi^{t,e}(s):=\e^{\int_{t+s}^t e(u)\, du}\psi (s)$, for $s\le 0, t\gg 1$ and $\psi \in C_g^0(\R)$, are well-defined, $i,j=1,\dots,n$;

(ii) for some vector $v>0$, $ \big[\widehat D(t)-\tilde A(t)-e(t)I\big ]v\ge 0$ for $t\gg 1$, where 
$\tilde A(t)=\Big [ \|\tilde L_{ij}(t)\|\Big ]$ and $I$ is the $n\times n$  identity matrix;

(iii)   $\int_0^\infty e(t)\, dt=\infty$.\\
Then
\eqref{Lin3} is  asymptotically  stable. In other words, \eqref{Lin3} is stable (on some interval $[\al,\infty)$) and all its solutions  satisfy 
$\lim_{t\to \infty}x(t)=0.$ Moreover, if  conditions (i),~(ii) are satisfied with $e(t)\equiv \de>0$, then \eqref{Lin3} is exponentially asymptotically stable.
\end{lem}

\begin{proof}  With $e(t)\ge 0$, (ii) implies (H2), thus  \eqref{Lin3} is stable. Set $\tilde a_{ij}(t)= \|\tilde L_{ij}(t)\|$, for $\tilde L_{ij}(t)$ in \eqref{tildeL}.  Without loss of generality, take $v={\bf 1}$ in (ii) (see Remark \ref{rmk3.1}) and  $T\ge 0$, so that 
 \begin{equation}\label{3.9lem2} 
 d_i(t)-\sum_{j\ne i}|d_{ij}(t)|-\sum_{j}\tilde a_{ij}(t)-e(t)\ge 0,\q t\ge T,\, i=1,\dots,n.
 \end{equation}
Effect the change of variables $y(t)=\e^{E(t)}x(t)$, where $E(t)=\int_0^t e(u)\, du$. 
The linear DDE \eqref{Lin3}  is transformed into
$$y_i'(t)=-(d_i(t)-e(t))y_i(t)-\sum_{j\ne i}d_{ij}(t)y_j(t)+\sum_{j} \tilde L_{ij}(t)(y_{j,t}),\  i=1,\dots,n,\ t\ge 0.$$
In virtue of  \eqref{3.9lem2},  this transformed system satisfies (H2). Let $x(t)=x(t,t_0,\phi)$ be a solution of the original equation.
 From Lemma \ref{lem3.1}, it follows that
$|y(t)|=|y(t, t_0,\e^{E(\cdot)} \phi)|\le \|\e^{E(\cdot)} \phi\|_g\le \|\phi\|_g$ for $t\ge t_0\ge T$, thus  $|x(t,t_0,\phi)|\le \e^{-E(t)}\|\phi\|_g$ for all $t\ge t_0\ge T$ and $\phi \in {\cal C}$. As $ \e^{-E(t)}\to 0$ as $t\to\infty$, then 
$\lim_{t\to\infty} x(t)=0$ for all solutions of \eqref{Lin3}. With $E(t)=\de t$ for some $\de>0$, we obtain
$|x(t,t_0,\phi)|\le \e^{-\de t}\|\phi\|_g$ for  $t\ge t_0\ge T$ and $\phi \in {\cal C}$, and \eqref{Lin3} is exponentially asymptotically stable.
\end{proof}

For future reference, we mention that a closer look to the proof of Lemma \ref{lem3.2} shows that the requirement of having $e(t)$ a nonnegative function is only used to derive that \eqref{Lin3} is stable, since with $e(t)\ge 0$ condition (ii) in Lemma \ref{lem3.2} implies (H2), and that $\|\e^{E(\cdot)} \phi\|_g\le \|\phi\|_g$. For the case of finite delays, condition (i) above always holds and,  to derive only the global attractivity of the zero solution, $e(t)$ need not be nonnegative.

In what follows, we shall assume some additional requirements on  $D(t),L(t)$, in order to have conditions (i)-(iii) of Lemma \ref{lem3.2} satisfied.
For simplicity,   we write 
\begin{equation*}
 L_{ij}(t)(\psi)=a_{ij}(t)\int_{-\infty}^0\psi(s)\, d_s\nu_{ij}(t,s) ,\q t\ge 0, \psi\in C_g(\R),
\end{equation*}
where the functions 
$\nu_{ij}(t,s)$ are measurable in $(t,s)\in \R^+\times(-\infty,0]$, continuous in $t\ge 0$, left-continuous in $s \in (-\infty,0)$, with $g(s)\nu_{ij}(t,s)$ of bounded variation in $s\in (-\infty,0]$  and $\nu_{ij}(t,s)$ normalized (relative to the norm $\|\cdot\|_g$ in $C_g(\R)$), so that 
\begin{equation}\label{3.2'}a_{ij}(t):=\|L_{ij}(t)\|\q {\rm and}\q \int_{-\infty}^0g(s)\, d_s|\nu_{ij}|(t,s)=1.
\end{equation}
Therefore,
 nonautonomous  linear   systems \eqref{Lin3}  are written in a more descriptive  way as 
    \begin{equation}\label{Lin}
x_i'(t)=-\sum_{j=1}^n d_{ij}(t)x_j(t)+\sum_{j=1}^na_{ij}(t) \int_{-\infty}^0 x_j(t+s)\, d_s\nu_{ij}(t,s),\  i=1,\dots,n,\ t\ge 0.
\end{equation} 
Note that each component $L_i(t)x_t=\sum_{j=1}^n L_{ij}(t)x_{j,t}$ may contain terms with several  time-dependent discrete  delays, as well as distributed delays. In particular, each $L_{ij}(t)x_{j,t}$ may be of the form
 \begin{equation}\label{Li+discrete}
L_{ij}(t)x_{j,t}=-\sum_{k=1}^p d_{ij}^k(t)\ x_j(t-\tau_{ij}^k(t))+\al_{ij}(t) \int_{-\infty}^0 x_j(t+s)\, d_s\nu_{ij}(t,s),
 \end{equation} 
 or  \begin{equation}\label{Li+distrib}
L_{ij}(t)x_{j,t}=-\sum_{k=1}^p d_{ij}^k(t)\int_{-\tau_{ij}^k(t)}^0 x_j(t+s)\, d_s\xi_{ij}^k(t,s)+\al_{ij}(t) \int_{-\infty}^0 x_j(t+s)\, d_s\nu_{ij}(t,s),
 \end{equation} 
  with  $s\mapsto \xi_{ij}^k(t,s),\nu_{ij}(t,s)$ normalized  so that 
$ \int_{-\tau_{ij}^k(t)}^0g(s)\, d_s|\xi_{ij}^k|(t,s)= \int_{-\infty}^0g(s)\, d_s|\nu_{ij}|(t,s)=1$  for all $i,j,k$. In this case, $\|L_{ij}(t)\|\le  \sum_{k=1}^p |d_{ij}^k(t)|+|\al_{ij}(t)|$.
So far, our approach  does not require any special treatment of the terms with   delays $\tau_{ij}^k(t)$ in \eqref{Li+discrete} or \eqref{Li+distrib} --  although the last results of this section concern systems with $L_{ij}(t)$ as in   \eqref{Li+distrib} with $\al_{ij}\equiv 0$ and possible unbounded delays $\tau_{ij}^k(t)$.
Linear systems where the terms with  (either discrete  or distributed)  {\it  finite diagonal delays} play an important role, and thus are separated from the terms with unbounded delay, will be analyzed in Section 4.

 With this notation, in the sequel one or more of the assumptions below will be imposed:
 \begin{itemize}
\item[(H3)] there exist $\al_0>0$ and functions $\mu_{ij}:(-\infty,0]\to [0,\infty)$ 
such that $|\nu_{ij}(t,s)|\le |\mu_{ij}(s)|$ for all $t\ge 0,s\le 0$ and 
$ \int_{-\infty}^0 \e^{-\al_0 s}g(s) \, d |\mu_{ij}|(s)<\infty,\ i,j=1,\dots,n;$
\item[(H4)]  there exist  vectors $u,v> 0$  and $T\ge 0$ such that $M(t)v\ge u$ for $t\ge T$;
\item[(H5)]  there exist a vector $v> 0$, $\al>1$ and $T\ge 0$ such that $\widehat D(t)v\ge \al A(t)v$ for  $t\ge T$.
\end{itemize}


 Some comments about these hypotheses  follow.
 
\begin{rmk}\label{rmk3.2.0} {\rm   Observe that  (H3) is trivially satisfied by linear equations with finite delay $\tau$, since one may take $\nu_{ij}(t,s)\equiv \nu_{ij}(t,-\tau)$ for $s\le -\tau$. On the other hand, 
if $g(s)=\e^{-\ga s}\ (s\le 0)$ for some $\ga >0$ and ${\cal C}=C_\ga^0$ (or ${\cal C}=C_\ga$ as defined in Section 2), it is clear that $C_\ga^0\subset C_{\ga_1}^0$ for any $\ga_1>\ga$. Therefore, in this case (H3) implies that the operators $L_{ij}(t)\in (C_\ga^0(\R))'$ are uniformly bounded for $t\ge 0$ by the operator $T_{ij}$ in $ (C_{\ga_1}^0(\R))'$ (identified with the measure $\mu_{ij}$), for some $\ga_1=\ga+\al_0>\ga$.}
  \end{rmk}

 \begin{rmk}\label{rmk3.2} {\rm  Both (H4) and (H5) are stronger versions of (H2), and they are equivalent under some boundedness conditions for the coefficients.
  In fact, if all the functions $a_{ij}(t)$ are bounded, then hypothesis (H4) implies (H5), since, with $v,u> 0$ as in (H4), then (H5) holds with the same vector $v$ and any $1< \al<1+\min_i(u_iM_i^{-1})$, where $0\le \sum_{j=1}^na_{ij}(t)v_j\le M_i$ in $[T,\infty)$. Similarly, if  all   the functions 
  $(\widehat D(t)v)_i$ are bounded from below by a positive constant for $t>0$ large (which is clearly satisfied if (H5) holds and the functions $\sum_ja_{ij}(t)$ are all  bounded from below by a positive constant), then hypothesis (H5) implies (H4): we have $(\widehat D(t)v)_i-\al\left(\sum_{j=1}^na_{ij}(t)v_j\right)\ge 0$ and $(\widehat D(t)v)_i\ge m_i>0$, which implies (H4) with the same vector $v$ and $u=(1-\al^{-1})(m_1,\dots,m_n)$. In particular, if there are $m,M>0$ such that $m\le \sum_{j=1}^na_{ij}(t)\le M$ for all $i$ and $t$ large, assumptions (H4) and (H5) are equivalent.}
  \end{rmk}


\begin{lem}\label{lem3.3} For  $\nu_{ij}$ continuous, with $\nu_{ij}(t,\cdot)\in M_g((-\infty,0];\R)$ and $\int_{-\infty}^0g(s)\, d_s|\nu_{ij}|(t,s)=1$, assume (H3). Then, for any $\eta>0$ there exists $\de>0$ such that, for $i,j=1,\dots,n$,
 \begin{equation}\label{3.8}
 \int_{-\infty}^0 \e^{-\de s}g(s)\, d_s |\nu_{ij}|(t,s)<1+\eta\q {\rm for\ all}\q t\ge 0.
 \end{equation}
\end{lem}

 \begin{proof}   We prove \eqref{3.8} for each  $i,j$ fixed. 
 Define
 $F(t,\al)=\int_{-\infty}^0 \e^{-\al s} g(s)\, d_s |\nu_{ij}|(t,s)$. By (H3), $F(t,\al)$ is well-defined for $(\al,t)\in[0,\al_0]\times [0,\infty)$ and 
 $${\bf F}(\al):=\sup_{t\ge 0}F(t,\al)\le \int_{-\infty}^0 \e^{-\al_0 s} g(s) \, d |\mu_{ij}|(s)=:C<\infty,\q \al\in [0,\al_0].$$
 Moreover, $\al\mapsto F(t,\al)$ ($t\ge 0$) and $ {\bf F}(\al)$ are non-decreasing in $\al\in [0,\al_0]$, with $F(t,0)={\bf F}(0)=1$.

If $C=1$, then ${\bf F}\equiv 1$ and  \eqref{3.8} holds for all $\de\in [0,\al_0]$. Otherwise, for any $\vare >0$  given,
 choose $M>0$ such that
$$\int_{-\infty}^{-M} \e^{-\al_0 s} g(s)\, d |\mu_{ij}|(s)<\vare/3.$$
Since $f(s,\al):=\e^{-\al s}$ is uniformly continuous on $[-M,0]\times [0,\al_0]$, there exists $\sigma>0$ such that
$|\e^{-\al s}-\e^{-\be s}|<\vare/3$ for any $(s,\al),(s,\be)\in [-M,0]\times [0,\al_0]$ with $|\al-\be|<\sigma$. Thus for $t\ge 0$ and $\al,\be \in  [0,\al_0]$ with $|\al-\be|<\sigma$,
$$|F(t,\al)-F(t,\be)|\le 2\int_{-\infty}^{-M} \e^{-\al_0 s} g(s) \, d |\mu_{ij}|(s)+\int_{-M}^0 |\e^{-\al s}-\e^{-\be s}|\, g(s)\, d_s |\nu_{ij}|(t,s)<\vare.$$
This estimate and the monotonicity properties of ${\bf F}$ and $F(t,\cdot)$ imply that $|{\bf F}(\al)- {\bf F}(\be)|\le \vare$ if $|\al-\be|<\sigma$.
This shows that ${\bf F}$ is continuous on $[0,\al_0]$.  Therefore for any $\eta>0$ with $1+\eta \le C$, there exists $\de \in (0,\al_0)$ such that
$$\int_{-\infty}^0 \e^{-\de s}g(s)\, d_s |\nu_{ij}|(t,s)\le {\bf F}(\de) =1+\eta,\q t\ge 0.$$
 The proof is complete. \end{proof}


We are ready to state the main results of this section.
We first address the exponential stability of \eqref{Lin}. For $D(t)=[d_{ij}(t)]$ with $d_{ij}(t)$ continuous and {\it bounded}, it follows e.g. from \cite[Proposition 6.3]{Coppel} that the ODE 
  $x'(t)=-D(t)x(t)$  is  exponentially  asymptotically stable if (H4) holds, i.e., if  for $ \widehat  D(t)$ as in \eqref{3.7}  there are positive vectors $v,u$ such that $\widehat  D(t)v\ge u$ for $t\gg 1$.
The generalization of this result to DDEs \eqref{Lin} is the subject of the next theorem, and does not require the a priori boundedness of all the coefficients.

\begin{thm}\label{thm3.1}  For  system \eqref{Lin}, assume (H1), (H3), and one of the following sets of conditions:
  \vskip 0mm
(i) (H4) is satisfied and $a_{ij}(t)$ are bounded functions on $\R^+$ for all $i,j=1,\dots,n$;
\vskip 0mm
(ii)  (H5) is satisfied and, for $v=(v_1,\dots,v_n)>0$ as in (H5),
$\liminf_{t\to\infty}(d_{ii}(t)v_i-\sum_{j\ne i}|d_{ij}(t)|v_j)>0$ for $ i=1,\dots,n$.\\
Then, \eqref{Lin3} is exponentially asymptotically stable.  \end{thm}
    
   \begin{proof}  (i) Denote $d_i(t):=d_{ii}(t)$ and $a_{ij}(t)$ as in \eqref{3.2'}. After rescaling the variables we take $v={\bf 1}$ in (H4), and consider $T,m,M>0$ such that $d_i(t)-\sum_{j\ne i}|d_{ij}(t)|-\sum_{j}a_{ij}(t)\ge m$ and $\sum_{j}a_{ij}(t)\le M$, for all $t\ge T,\, i=1,\dots,n$. For $0<\eta <m(1+M)^{-1}$, we have 
 \begin{equation}\label{3.9} 
 d_i(t)-\eta-\sum_{j\ne i}|d_{ij}(t)|-(1+\eta)\sum_{j}a_{ij}(t)>0,\q t\ge T,\, i=1,\dots,n.
 \end{equation}
Next, by Lemma \ref{lem3.3}, choose $\de\in (0,\eta)$ such that \eqref{3.8} holds. With $e(t)=\de$,  the operators $\tilde L_{ij}(t)$ in \eqref{tildeL},  given by
   \begin{equation}\label{3.10} 
\tilde L_{ij}(t)\phi_j= L_{ij}(t)(\e^{-\de \cdot}\phi_j)=a_{ij}(t)\int_{-\infty}^0 \e^{-\de s}\phi_j(s)\, d_s\nu_{ij}(t,s)
\end{equation}
 for all $i,j$ and $t\ge 0,\phi=(\phi_1,\dots,\phi_n)\in {\cal C}$, are well defined; moreover, $\|\tilde L_{ij}(t)\|\le (1+\eta) a_{ij}(t)$. 
 From Lemma \ref{lem3.2}, it follows that \eqref{Lin3} is exponentially asymptotically stable.

(ii) As above,  take $v={\bf 1}$ in (H5). Choose $m>0$ such that $d_i(t)-\sum_{j\ne i}|d_{ij}(t)|\ge m$ for $t\ge T,1\le i\le n$,  and $\eta >0$ with $1+\eta <\al$. Next, choose $\de >0$ such that \eqref{3.8} holds  and
$\de <m[1-(1+\eta)\al^{-1}]$. We obtain \begin{equation*}
\begin{split}
 d_i(t)-\de-\sum_{j\ne i}|d_{ij}(t)|-(1+\eta)\sum_j  a_{ij}(t)&\ge [1-(1+\eta)\al^{-1}]\Big (d_i(t)-\sum_{j\ne i}|d_{ij}(t)|\Big)-\de \\
&\ge  m[1-(1+\eta)\al^{-1}]-\de>0,
\end{split}
\end{equation*}
and again the result follows from Lemma \ref{lem3.2} with $e(t)=\de$.
\end{proof}

A closer look to  the proof of (i) above shows that, in the case of finite delays, the functions  $a_{ij}(t)=\|L_{ij}(t)\|$  are not required to be bounded. 

\begin{cor}\label{cor3.1} Assume (H1), (H3), $\liminf_{t\to\infty}d_{ii}(t)>0$ and  that  there exist  $v=(v_1,\dots,v_n)> 0$, $T\ge 0$ and $\al>1$ such that 
$d_{ii}(t)v_i\ge \al \Big(\sum_j  \big [(1-\de_{ij})|d_{ij}(t)|+  a_{ij}(t)\big ]v_j\Big)$ for $ t\ge T,  i=1,\dots,n,$
  where $\de_{ij}=1$ if $i=j$, $\de_{ij}=0$ if $i\ne j$. Then \eqref{Lin} is exponentially asymptotically stable.
\end{cor}

\begin{proof} This is a particular case of (ii), if in \eqref{Lin} we take $d_{ij}(t)=0$  and replace $L_{ij}(t)\psi $ by $d_{ij}(t)\psi(0)+L_{ij}(t)\psi$, for all $j\ne i$.
\end{proof}

Consider now the case of {\it autonomous} linear DDEs of the form
    \begin{equation}\label{LinAu}
    x_i'(t)=-Dx(t)+Lx_t,\q t\in I,
\end{equation} 
where $D=[d_{ij}]\in \R^{n\times n}$ with $d_i:=d_{ii}>0$, and $L=(L_1,\dots,L_n)\in L({\cal C},\R^n)$. Of course, for  \eqref{LinAu} the asymptotic and exponential stabilities coincide.
As before, write the components $L_i$ of $L$ as $L_i(\phi)=\sum_{j=1}^n L_{ij}(\phi_j)$, define   the $n\times n$ matrices
  \begin{equation}\label{MatrAu}
  D=[d_{ij}],\q \widehat D=[\widehat d_{ij}],\q A=[a_{ij}],\q M=\widehat D-A,
  \end{equation} 
where $\widehat d_{ii}=d_{ii}, \widehat d_{ij}=-|d_{ij}|$ if $i\ne j$ and $a_{ij}=\|L_{ij}\|$, for $ i,j=1,\dots,n$. 
 In this situation, condition (H2) simply says that $Mv\ge 0$ for some positive vector $v$.
From Lemma \ref{lem3.1},  it follows that (H2) implies that \eqref{LinAu} is stable, thus all the roots of its characteristic equation have nonpositive real parts. It is easy to see that
 (H2) is not sufficient to guarantee that \eqref{LinAu} is exponentially stable in any space ${\cal C}$ (see e.g.~Example \ref{exmp2.4} in the last section).
On the other hand,  the next theorem asserts that, if 0 is not a characteristic value, then asymptotic stability follows  with $M$ satisfying a property  weaker than (H2). To show this, some algebraic definitions and properties are recalled below.
 
 \begin{defn}\label{defn2.2}   A square matrix $N=[n_{ij}]$  with nonpositive off-diagonal entries   (i.e., $n_{ij}\le 0$ for $i\ne j$) is said to be an {\bf  M-matrix}, respectively a {\bf non-singular M-matrix}, if   all its eigenvalues have non-negative, respectively positive, real parts. 
\end{defn}

If $n_{ij}\le 0$ for $i\ne j$, it is well-known that  $N=[n_{ij}]$  is a non-singular M-matrix if and only if there exists a positive vector $v$ such that $Nv>0$; and if there is a vector $v\ge 0$ such that $Nv\ge 0$, then $N$ is an M-matrix; the converse is not always true (but it is valid in the case of irreducible matrices).
See e.g. \cite{Berman}, also for further properties of these matrices.

\begin{thm}\label{thm3.2}  Consider in ${\cal C}$ the autonomous linear system 
$x'(t)=-Dx(t)+Lx_t,$
where $D=[d_{ij}]_{n\times n}$ and $L=(L_1,\dots,L_n)\in L({\cal C};\R^n)$. For $L_i(\var)=\sum_jL_{ij}(\var_j)$ for $\var=(\var_1,\dots,\var_n)\in {\cal C}$, define   the matrices
\begin{equation}\label{MM}
M_0=-D+\Big [L_{ij}(1)\Big],\q M=\widehat D-A,
\end{equation}
where $D,A, M$ are as in \eqref{MatrAu}.
 If $\det M_0\ne 0$ and $M$ is an M-matrix, then \eqref{LinAu} is (exponentially)  asymptotically stable.  \end{thm}

This theorem  generalizes the criterion obtained in \cite[Theorems 2.3 and 2.6]{FO08},
where $\det M_0\ne 0$ and $M$  an M-matrix were proven to be sharp conditions for the {\it absolute}    exponential stability of  {autonomous} linear DDEs with {\it finite delays}. Although we are considering infinite delays in $x'(t)=-Dx(t)+Lx_t$, the proof of Theorem \ref{thm3.2} follows along  arguments  similar to the ones in  \cite{FO08}, and is therefore omitted.
  
\begin {rmk}\label{rmk3.3}  {\rm In \cite{NgocCao16},
  Ngoc and Cao  investigated the exponential stability  of  a nonautonomous linear system  in $C_\ga$ (for some $\ga>0$) written in the abstract form
\begin{equation}\label{LinNT}
x'(t)=-D_0(t)x(t)+\sum_{k=1}^\infty D_k(t)x(t-\tau_k(t))+\int_{-\infty}^0B(t,s)x(t+s)\, ds,
\end{equation}
where the discrete delays $\tau_k(t)\ge 0$ are all uniformly {\it bounded} by positive constants $\tau_k$,  the $n\times n$ matrices of   functions $D_k(t)=[d_{ij}^k(t)], B(t,s)$
are all bounded by {\it autonomous} matrices $D_k, B(s)$ for all $k\in\N_0,t\ge 0,s\le 0$,  with the exception of the diagonal entries $d_{ii}^0(t)$ of $D_0(t)$, which are only required to be bounded from below by a positive constant $c_i>0, i=1,\dots,n$. 
 Note that  \eqref{Lin3} encompasses systems of the form \eqref{LinNT}.
Ngoc and Cao \cite[Theorem 3.3]{NgocCao16} derived the exponential stability of \eqref{LinNT} in some space $C_{\ga_0}$ with $\ga_0\in (0,\ga)$ under additional  stronger conditions. Besides  the uniform bounds of the entries of matrices $D_k(t),B(t,s)$ as described above, it was further imposed:   a hypothesis with the role of the present assumption  (H3),
$\sum_{k=1}^\infty \e^{\ga \tau_k}\|D_k\|<\infty$, and that there exists an $n\times n$ {\it non-singular M-matrix} of constants $M$  such that,
with our notations, $M(t)\ge M$ for all $t\ge 0$. This latter requirement  is stronger than either (H4) or (H5). More recently, Ngoc et al. \cite{Ngoc19}  considered linear DDEs \eqref{Lin3} but only with {\it finite} delay, without  requiring  a priori uniforms bounds of all the coefficients;  for the situation of finite delay, Theorem 3.2.(iv) in \cite{Ngoc19} is exactly
 the criterion expressed in our Theorem 3.1.(i), nevertheless    the other criteria in \cite[Theorem 3.2.(i)-(iii)]{Ngoc19} are more restrictive than the ones in our Theorem 3.1, even for the situation of finite delay.}
\end{rmk} 
 We now establish criteria  for the asymptotic stability  (but not necessarily  exponential stability) of \eqref{Lin}
 without imposing (H3), nor that part of the coefficients are bounded.
  Therefore, we need to restrict  the class  \eqref{Lin3}, in order to guarantee the existence of some function $e(t)$ satisfying the assumptions  in Lemma \ref{lem3.2}. Henceforth, in this section we treat linear equations \eqref{Lin} of the  particular form
  \begin{equation}\label{LinTau}
x_i'(t)=-\sum_{j=1}^n d_{ij}(t)x_j(t)+\sum_{j=1}^na_{ij}(t) \int_{-\tau_{ij}(t)}^0 x_j(t+s)\, d_s\nu_{ij}(t,s),\  i=1,\dots,n,\ t\ge 0,
\end{equation}
for some nonnegative, continuous and possibly {\it unbounded} delay  functions $\tau_{ij}(t)$.  With the previous notation, we may suppose that  $\nu_{ij}(t,\cdot):[-\tau_{ij}(t),0]\to \R$ are bounded variation functions and satisfy
\begin{equation}\label{3.16}
\int_{-\tau_{ij}(t)}^0 g(s)\, d_s|\nu_{ij}|(t,s) =1,\q t\ge 0, i,j=1,\dots,n.
\end{equation}
Alternatively, we may still consider $\nu_{ij}(t,\cdot):(-\infty,0]\to\R$ as in  \eqref{3.2'} and take $\nu_{ij}(t,s)=\nu_{ij}(t,-\tau_{ij}(t))$ for $s<-\tau_{ij}(t)$ and $t\ge 0$.
For \eqref{LinTau}, Lemma \ref{lem3.2} gives the criterion below.

\begin{thm}\label{thm3.3} Consider \eqref{LinTau} with  $\tau_{ij}:[0,\infty)\to [0,\infty)$ continuous,  and set $\tau(t)=\max_{1\le i,j\le n}\tau_{ij}(t).$ Assume  (H1), where now the bounded variation functions $\nu_{ij}(t,s)$ satisfy \eqref{3.16}.
In addition, suppose that $\tau(t)\le t$ for $t\gg 1$ and that there exists a  measurable, locally integrable function $e:\R\to \R^+$, a vector $v>0$ and $T\ge 0$,  such that  the following conditions are satisfied:

(i) $ \big[\widehat D(t)-\e^{\int_{t-\tau(t)}^t e(u)\, du}A(t)-e(t)I\big ]v\ge 0$ for $t\ge T$, where $I$ is the $n\times n$  identity matrix;

(ii)   $\int_0^\infty e(t)\, dt=\infty$.\\
Then
\eqref{LinTau} is  asymptotically  stable.\end{thm}

\begin{proof} The operators in \eqref{tildeL} are given by $\tilde L_{ij}(t)\var=a_{ij}(t) \int_{-\tau_{ij}(t)}^0 \var(s) \e^{ \int_{t+s}^t e(u)\, du}\, d_s\nu_{ij}(t,s)$ (for $\var\in C_g^0(\R)$), thus we have  the estimates
\begin{equation}\label{2.12}\tilde a_{ij}(t):=\| \tilde L_{ij}(t)\|\le a_{ij}(t)\e^{ \int_{t-\tau_{ij}(t)}^t e(u)\, du} ,\q {\rm }\q t\gg 1,i,j=1,\dots,n.
\end{equation}
\end{proof}

\begin{rmk}\label{rmk3.4} {\rm It is clear that in (i) the matrix $\e^{\int_{t-\tau(t)}^t e(u)\, du}A(t)$
can be replaced by the matrix $\Big [a_{ij}(t)\e^{ \int_{t-\tau_{ij}(t)}^t e(u)\, du}\Big]$. Note also that condition (i) expresses a restriction on the size of the delay functions $\tau_{ij}(t)$, nevertheless  the delays need not be bounded (see e.g. Example 5.2 in Section 5).}
\end{rmk}


\begin{thm}\label{thm3.4} Under the general notations in Theorem \ref{thm3.3} with $\tau(t)\le t$ for $t$ large, assume (H1), and  (H5).  For some $v> 0$ as in  (H5) and some $T\ge 0$,
suppose also that there exists a  measurable, locally integrable function $e:\R\to \R^+$  such that:

(i) $\min_{1\le i\le n}\left (M(t)v\right)_i\ge e(t),t\ge T$;

(ii)   $\int_0^\infty e(t)\, dt=\infty$;

(iii) $\sup_{t\ge T}\int_{t-\tau(t)}^t e(u)\, du<\infty.$\\
Then
\eqref{LinTau} is asymptotically  stable.\end{thm}

\begin{proof} If (H5) holds, take a vector $v=(v_1,\dots,v_n)> 0$, a constant $\al>1$   and $T\ge 0$ such that 
\begin{equation}\label{2.13'}
d_i(t)v_i-\sum_{j\ne i} |d_{ij}(t)|v_j\ge \al  \Big (\sum_j a_{ij}(t)v_j\Big)\q {\rm for}\q 1\le i\le n, t\ge T.
\end{equation}
Hence $e_i(t):=d_i(t)v_i-\sum_{j\ne i} |d_{ij}(t)|v_j-\sum_j a_{ij}(t)v_j\ge (1-\al^{-1}) \Big (d_i(t)v_i-\sum_{j\ne i} |d_{ij}(t)|v_j\Big )$. Set  $E(t)=\de \int_0^t e(u)\, du$, with $\de>0$ sufficiently small. After a scaling, take $v={\bf 1}$. From (i) and \eqref{2.12},
 \begin{equation*}
 \begin{split}
 d_i(t)&-\de e(t)-\sum_{j\ne i} |d_{ij}(t)|- \sum_j \tilde a_{ij}(t)\ge
 d_i(t)-\sum_{j\ne i} |d_{ij}(t)|-\de e_i(t)-\e^{\de \int_{t-\tau(t)}^t e(u)\, du} \sum_j a_{ij}(t)\\
& \ge \left (1-\de (1-\al^{-1})-\al^{-1}\e^{\de \int_{t-\tau(t)}^t e (u)\, du}\right )\left ( d_i(t)-\sum_{j\ne i} |d_{ij}(t)|\right),\q t\ge T, i=1,\dots,n.
 \end{split}
 \end{equation*}
 From (iii), $1-\de (1-\al^{-1})-\al^{-1}\e^{\de \int_{t-\tau(t)}^t e (u)\, du}\to 1-\al^{-1}>0$ as $\de\to 0^+$, and the conclusion follows by the previous theorem.  \end{proof}
 
%

 \begin{rmk}\label{rmk3.5} {\rm 
If in each equation there is only a single {\it discrete} delay for each variable,
so that  \eqref{LinTau} reads as
 \begin{equation}\label{LinTauDisc}
x_i'(t)=-\sum_{j=1}^n d_{ij}(t)x_j(t)+\sum_{j=1}^n b_{ij}(t) x_j(t-\tau_{ij}(t)),\  i=1,\dots,n,\ t\ge 0,
\end{equation}
in \eqref{2.12} we have the identity $\tilde a_{ij}(t)=a_{ij}(t)\e^{ \int_{t-\tau_{ij}(t)}^t e(u)\, du}$, where $a_{ij}(t)=|b_{ij}(t)|$. Thus,  in the above theorem if we  do not require
the function $e(t)$   to be nonnegative,  as long as it satisfies (i)--(iii), we are still able to conclude that all solutions of \eqref{LinTauDisc} tend to zero at infinity. This result also applies to equations  with multiple  terms with {\it discrete} delays, whose coefficients are of the same sign, i.e., to equations
\begin{equation}\label{LinTauDisc2}
x_i'(t)=-\sum_{j=1}^n d_{ij}(t)x_j(t)+\sum_{j=1}^n \sum_{k=1}^p d_{ij}^k(t) x_j(t-\tau_{ij}^k(t)),\  i=1,\dots,n,\ t\ge 0,
\end{equation}
with $\Big |\sum_{k=1}^p d_{ij}^k(t) \Big |=\sum_{k=1}^p| d_{ij}^k(t)|$ for all $i,j,k$ and $t\ge 0$ large.
}
\end{rmk}

 Other criteria can be derived in a similar way. For instance, it is easy to verify that  (H5) can be eliminated in Theorem \ref{thm3.4}  if (i), (iii) are replaced by slightly stronger conditions, as follows: (i') $\min_{1\le i\le n}\left (M(t)]v\right)_i=:e_i(t)>0$ and $e_i(t)\ge e(t),t\ge T$; (iii') $\limsup_{t\to\infty} \big (\e^{\int_{t-\tau(t)}^t e(u)\, du}-1\big) \frac{\left (\widehat D(t)]v\right)_i}{e_i(t)}<\infty$ for $1\le i\le n$.

 The scalar case is addressed in the following corollary:

\begin{cor}\label{cor3.3} Consider  the scalar linear DDE
\begin{equation}\label{LinSc}
x'(t)=-d(t)x(t)+L_0(t)x_t,\q t\ge 0
\end{equation}
 with $d(t)> 0$ continuous and $L_0(t)\var=\be(t) \int_{-\tau(t)}^0\var (s)\, d_s\nu(t,s)$
for $ (t,\var)\in [0,\infty)\times  C_g^0(\R)$,
where $\tau(t)\ge 0, \nu(t,s)$ are continuous in $t$, $\nu(t,\cdot)\in M_g((-\infty,0];\R)$ with $ \int_{-\tau(t)}^0g(s)d_s|\nu|(t,s)=1$ so that
$\be(t)=\|L_0(t)\|$  for $t\ge 0$, and $\tau(t)\le t$ for $t\ge 0$ large.
Assume that there exist $T>0$ and a measurable, locally integrable function  $e:\R\to\R^+$ such that one of the following conditions holds:
 \begin{itemize}
\item[(a)] $d(t)-\be(t)\sup_{t\ge T}\e^{\int_{t-\tau(t)}^t e(u)\, du}\ge e(t)$  for $ t\ge T$, with $\int^\infty e(t)\, dt=\infty$;
 \item[(b)]  (i)  $d(t)\ge \al \be(t)$  for some $\al>1$ and $ t\ge T$;
 
  (ii) $e(t)$ satisfies
  $e(t)\le d(t)-\be(t)$, $\int^\infty e(t)\, dt=\infty$ and
 $\sup_{t\ge T}\int_{t-\tau(t)}^t e(u)\, du<\infty.$
 \end{itemize}
  Then \eqref{LinSc} is globally asymptotically stable. If in addition $d(t)\ge c$ for some constant $c>0$ and $t$ sufficiently large, then  \eqref{LinSc} is globally exponentially stable. \end{cor}

%

%

\begin {rmk}\label{rmk2.4} {\rm   In a  recent paper \cite{GH}, Gy\H ori and Horv\'ath studied Halanay-type nonautonomous delay differential inequalities of the form
\begin{equation}\label{2.16}
x'(t)\le -d(t)x(t)+\be (t)\sup_{s\in [t-\tau(t),t]}x(s),\q t\ge t_0,
\end{equation}
and
\begin{equation}\label{2.17}
x'(t)\le -d(t)x(t)+\be (t)x(t-\tau(t)),\q t\ge t_0,
\end{equation}
where $d,\be:[t_0,\infty)\to\R^+$ are locally integrable, $\tau:[t_0,\infty)\to\R^+$ is measurable with $t_0-\tau_0\le t-\tau(t)\to \infty$ as $t\to\infty$, for some $\tau_0\ge 0$. In fact, as mentioned,  this more general framework, of locally integrable coefficients $d(t)$, $\be(t)$ and a measurable delay  $\tau(t)$, could have been considered here.
By using a different approach, in  \cite{GH} the authors presented a comprehensive, refined analysis of
the global attractivity of the zero solution of \eqref{2.16} and \eqref{2.17} (although the attractivity in \cite{GH} only concerns the nonnegative solutions of such inequalities).  The main tools employed in \cite{GH} are generalized  Halanay-type inequalities and  the so-called {\it generalized characteristic equation}, applied to the nonautonomous scalar  differential equation
\begin{equation}\label{GHeq}
x'(t)= -d(t)x(t)+\be (t)x(t-\tau(t)),\q t\ge t_0,
\end{equation}
given by
$e(t)+\be(t)\e^{\int_{t-\tau(t)}^t e(s)\, ds}=d(t),$
where $e(t)$ is a locally integrable function on $[t_0-\tau_0,\infty)$.
As in the present paper, in \cite{GH} the coefficients $d(t),\be(t)$ are not required to be bounded,  a constraint often imposed in the  literature.
Among other  results, in particular Gy\H ori and Horv\'ath  gave a sharp criterion \cite[Theorem 2.8]{GH} as follows:
every nonnegative solution of \eqref{2.17} tends to zero at infinity if and only if condition (a) in Corollary \ref{cor3.3} is satisfied by some  measurable, locally integrable function $e(t)$, which however is not required to be nonnegative (cf. Remark \ref{rmk3.5} above). Thus, the sufficient condition of Gy\H ori and Horv\'ath's result is a simple consequence of our Corollary \ref{cor3.3} with $L_0(t)x_t=\be(t) x(t-\tau(t))$.
%
 In \cite[Theorem 3.6]{GH}, it is was also established that if there exists $\al>1$ such that $d(t)\ge \al \be (t)$ (for $t\ge t_0$) and $\sup_{t\ge 0}\int_{\max(t-\tau(t),t_0)}^t (d(u)-\be (u))\, du<\infty$, then the zero solution of  \eqref{2.16} or  \eqref{2.17} is globally attractive if  and only if $\int_{t_0}^\infty  (d(t)-\be (t))\, dt=\infty$. When applied to  \eqref{GHeq},
 these conditions read as (b) of Corollary \ref{cor3.3} with the choice $e(t)=d(t)-\be(t)$, thus, again, the criterion in \cite[Theorem 3.6]{GH} is a particular case of
  Corollary \ref{cor3.3}. Note that, not only \eqref{LinSc} is more general than the scalar equation with one single discrete delay, but also its nonnegative solutions satisfy \eqref{2.16} with $\be(t)=\|L_0(t)\|$.  For the $n$-dimensional case, whether the hypotheses in Theorems \ref{thm3.3} and \ref{thm3.4}  are optimal or not is an interesting question deserving future investigations. 
}\end{rmk}

\section{Linear systems with  pure diagonal delays}
\setcounter{equation}{0}

This section is devoted to the study of linear equations \eqref{Lin0} which do not necessarily have a domi\-nant   diagonal negative feedback term without delay in each equation. In these circumstances, we shall assume  the existence of one or several terms with {\it diagonal  finite delays}, which may be either discrete or distributed, but which nevertheless dominate the effect of both the off-diagonal terms and the diagonal terms with infinite delay. 
For DDEs with only finite discrete delays, see the recent paper \cite{BDSS19} for further interesting results on exponential stability depending on all  delays.

We go back to a general  linear DDE   \eqref{Lin0} in ${\cal C}$, and suppose that the terms with diagonal {\it finite} delays are separated from the others:
\begin{equation}\label{4.1}
x_i'(t)=-\sum_{k=1}^p d_{ii}^k(t)\ell_i^k(t)x_{i,t}+L_i(t) x_t,\q i=1,\dots,n,t\ge 0,
\end{equation}
for  bounded linear  functionals $\ell_i^k(t)\in (C([-\tau_{ii}^k(t),0];\R))'$ and $\tau_{ii}^k(t)\ge 0$  bounded delays, $d_{ii}^k:\R^+\to\R$ continuous, $k=1,\dots,p$, and,  as before,
$L(t)\in L({\cal C},\R^n)$ is given in coordinates by $L_i(t)x_t=\sum_{j=1}^n L_{ij}(t) x_{j,t}$, for $i=1,\dots, n$ and $t\ge 0$. Note that nondelayed terms may be included in \eqref{4.1}. An extra condition on the operators $\ell_i^k(t)$ will be imposed, but first further comments on the phase space ${\cal C}=C_g^0(\R^n)$ are given.

Since the delays $\tau_{ii}^k(t)$ are bounded, say $\tau_{ii}^k(t)\le r$ for all $t\ge 0,i=1,\dots, n,k=1,\dots,p$, the norms $\|\cdot\|_\infty$ and $\|\cdot\|_g$ are equivalent in $C([-\tau_{ii}^k(t),0];\R^n)$. Thus the space $M_g([-\tau_{ii}^k(t),0];\R)$ coincides with the usual space $(C([-\tau_{ii}^k(t),0];\R))'=BV([-\tau_{ii}^k(t),0];\R)$.
For convenience,  we write 
  \begin{equation}\label{T_i^k}
  \begin{split}
 \ell_i^k(t)\var&=\int_{-\tau_{ii}^k(t)}^0\var(s)\, d_s\xi_{ii}^k(t,s),\q t\ge 0,\var\in C_g(\R),
 \end{split}
 \end{equation}
for $\xi_{ii}^k(t,\cdot)\in BV([-\tau_{ii}^k(t),0];\R)$.
With \eqref{T_i^k} and the previous notation for $L(t)$,  \eqref{4.1} is given by
\begin{equation}\label{4.2}
\begin{split}
x_i'(t)=&-\sum_{k=1}^p d_{ii}^k(t)\int_{-\tau_{ii}^k(t)}^0 x_i(t+s)\, d_s\xi_{ii}^k(t,s)\\
&+\sum_{j=1}^na_{ij}(t) \int_{-\infty}^0 x_j(t+s)\, d_s\nu_{ij}(t,s),\q i=1,\dots,n,\ t\ge 0,
\end{split}
\end{equation}
where $a_{ij}(t), \nu_{ij}(t,s)$ are as in \eqref{3.2'}. 
In the sequel, the following   conditions are assumed:
  \begin{itemize}
\item[(H1*)]   (i) $\ell_i^k, L_{ij}:\R^+\to L(C_g^0(\R),\R)$ and $d_{ii}^k:\R^+\to\R$ are continuous  for all $i,j,k$;

(ii) the operators $\ell_i^k(t)$ are given by \eqref{T_i^k}, where $\tau_{ii}^k:\R^+\to[0,r]$ are continuous (for some $r>0$), and $\xi_{ii}^k(t,s)$ are measurable, continuous on $t$,  {\it nondecreasing} on $s$ and normalized so that
\begin{equation}\label{xi_normal}\xi_{ii}^k(t,0)-\xi_{ii}^k(t,-\tau_{ii}^k(t))=1,\q t\ge 0, 1\le i\le n, 1\le k\le p;
\end{equation}

 (iii) $d_i(t):=\sum_{k=1}^pd_{ii}^k(t)>0$ for $t\ge 0$ and $i=1,\dots,n$.


\end{itemize}

We give a few comments about this general hypothesis. Clearly, the condition that  $\xi_{ii}^k(t,s)$ are nondecreasing in $s\in [-\tau_{ii}^k(t),0]$ expresses that  the  functionals $ \ell_i^k(t)$ are {\it nonnegative} functionals \cite{Rudin}. Observe also that from (H1*)(ii)  we obtain 
$$\ell_i^k(t)(1)= \int_{-\tau_{ii}^k(t)}^0 \, d_s\xi_{ii}^k(t,s)=1, \q t\ge 0, 1\le i\le n, 1\le k\le p;$$
nevertheless the norm $\|\ell_i^k(t)\|=\|\ell_i^k(t)\|_g$ is given by $\|\ell_i^k(t)\|=\int_{-\tau_{ii}^k(t)}^0 g(s)\, d_s\xi_{ii}^k(t,s)$. 
Here, although  $d_i(t)=\sum_k d_{ii}^k(t)$ is required to be positive, each function $d_{ii}^k(t)$ may be either posi\-tive or negative, or change sign on $\R^+$.
  The case of discrete delays is  included above, i.e., one may have
 $\ell_i^k(t)\var=\var(-\tau_{ii}^k(t))$ for several or all $i,k$; moreover, with  $\tau_{ii}^k(t)\equiv 0$, in this situation $\ell_i^k(t)\var=\var(0)$.
  Thus, the situation with a diagonal term without delay is included in the present form \eqref{4.2}.
 Note also that the operators $L_{ij}(t)$ in \eqref{4.1} may incorporate several terms with  bounded  delays, either discrete or distributed, as in \eqref{Li+discrete} or \eqref{Li+distrib}. In other words, this framework encompasses linear DDEs of the form
 \begin{equation}\label{4.2'}
\begin{split}
x_i'(t)=\sum_{j=1}^n&\bigg [-d_{ij}^0(t)x_j(t)-\sum_{k=1}^p d_{ij}^k(t)\int_{-\tau_{ij}^k(t)}^0 x_j(t+s)\, d_s\xi_{ij}^k(t,s)\\&+\al_{ij}(t) \int_{-\infty}^0 x_j(t+s)\, d_s\nu_{ij}(t,s)\bigg ],\q i=1,\dots,n,\ t\ge 0.
\end{split}
\end{equation}
 However only the {\it diagonal zero or  bounded} delays have a relevant  role in the results below.

Define  the $n\times n$ matrix-valued functions
\begin{equation}\label{4.3}
\begin{split}
D_k(t)= \diag\, (d_{11}^k(t),\dots ,d_{nn}^k(t))\ (1\le k\le p),& \q
D(t)=\diag\, (d_1(t),\dots,d_n(t)),\\
\q A(t)=\big[\|L_{ij}(t)\|\big ],&\q
  C^\tau(t)=\diag\, (c_1^\tau(t),\dots, c_n^\tau(t)),
  \end{split}
\end{equation}
and
\begin{equation}\label{M(t)}
M(t)=D(t)-C^\tau(t)-A(t),\q t\ge 0,
\end{equation}
where 
\begin{equation*}
\begin{split}
d_i(t)&=\sum_{k=1}^pd_{ii}^k(t),\\%
c_i^\tau(t)&
= \sum_{k=1}^p |d_{ii}^k(t)|\int_{t-\tau_{ii}^k(t)}^t   \left ( \sum_{l=1}^p|d_{ii}^l(u)|g(-\tau_{ii}^l(u))+\sum_{j=1}^n \|L_{ij}(u)\|\right )du,\q i=1,\dots,n.
\end{split}
\end{equation*}

%
%
%
%
%
%
%
%
%
%

  As mentioned, the situation with  $\ell_i^k(t)\var=\var(0)$  for some  or all $i,k$ is included in \eqref{4.1}. 
   Clearly, for systems of the form \eqref{Lin3} (that is, of the form \eqref{4.2'} with $ d_{ij}^k\equiv 0$ for all $i,j$ and $k=1,\dots,p$), then $C^\tau(t)\equiv 0$ and  the matrix $M(t)$ reduces to $M(t)=\widehat D(t)-A(t)$ as in \eqref{3.7}.

We now extend Lemma \ref{lem3.1} to equations of the form \eqref{4.1}, but beforehand we  remark that the weight function $g$ of the phase space $C_g^0$ satisfies a useful property: since ${\lim_{u\to 0^-}{{g(s+u)}\over {g(s)}}=1}$ uniformly on $(-\infty ,0]$, it follows that the set $\Big\{ \frac{g(s-r)}{g(s)}: s\le 0\Big\}$ is bounded, for any $r>0$.

%
%


\begin{lem}\label{lem4.1} Consider the system  \eqref{4.1} in the phase space ${\cal C}$. If  assumption (H1*) holds and the matrix $M(t)$ in \eqref{M(t)} satisfies (H2), then there exist $m\ge1, T_1\ge 0$ and a vector $v>0$ such that
the solutions $x(t)$ of \eqref{4.1} satisfy $|x(t)|_{v^{-1}}\le m  \|x_{t_0}\|_{g,v^{-1}}$ for $ t\ge t_0\ge T_1$.
 In particular, 
  \eqref{4.1}  is (uniformly) stable on $[T_1,\infty)$. \end{lem}

\begin{proof}  We adapt the proof of Lemma \ref{lem3.1},  stressing  however that the treatment of the delays $\tau_{ii}^k(t)\in[0,r]$  requires special care. Let  $T,v$ be as in (H2) and  $T_1 :=\max \{2r,T\}$.

 Define $m=\sup_{s\le 0}\frac{g(s-r)}{g(s)}.$ Clearly $m\ge 1$. After  rescaling the variables by $\bar x_i(t)=v_i^{-1}x_i(t)\, (1\le i\le n)$,  assume  (H2)  with $v={\bf 1}$ and take the norm
$|x|=\max_{1\le i\le n} |x_i|$. Note that $t-\tau_{ii}^k(t)\ge 0$ for $t\ge r$ and all $i$. Fix $\phi\in {\cal C}$, $ t_0\ge T_1$, consider the solution $x(t)=x(t,t_0,\phi)$ of \eqref{4.1} and  $J=[t_0,t_0+a]$, for any $a>0$. 
We  claim that
 \begin{equation}\label{claim}|x(t)|\le \ell:= m \|x_{t_0}\|_{g}\q {\rm for}\q t\in J.
 \end{equation}
 
 If this assertion fails to be true, 
 there exists $i\in \{1,\dots,n\}$ and $ t_1\in (t_0,t_0+a]$ such that $u_i:=|x_i(t_1)|=|x(t_1)|=\max_{t\in J}|x(t)|>\ell$.
Arguing as in the proof of Lemma \ref{lem3.1}, we derive that
$ \|x_{t}\|_g\le u_i$ for $t\in J.$
 Next, suppose that $x_i(t_1)>0$  (the case $x_i(t_1)<0$ is analogous). 
From (H1*), we obtain 
  \begin{equation}\label{ineq_1}
\begin{split}
x_i'(t)+d_i(t)x_i(t)&= \sum_{k=1}^p d_{ii}^k(t)\Big (x_i(t)-\int_{-\tau_{ii}^k(t)}^0 x_i(t+s)\, d_s\xi_{ii}^k(t,s)\Big )+\sum_{j=1}^n L_{ij}(t)x_{j,t}\\
&= \sum_{k=1}^p d_{ii}^k(t)\int_{-\tau_{ii}^k(t)}^0  (x_i(t)-x_i(t+s))\, d_s\xi_{ii}^k(t,s)+\sum_{j=1}^n L_{ij}(t)x_{j,t}\\
&\le  \sum_{k=1}^p d_{ii}^k(t)\int_{-\tau_{ii}^k(t)}^0 (\int_{t+s}^tx_i'(u)\, du) \, d_s\xi_{ii}^k(t,s)+u_i\sum_{j=1}^n\|L_{ij}(t)\|\\
&= \sum_{k=1}^p d_{ii}^k(t)\int_{t-\tau_{ii}^k(t)}^t x_i'(u)\bigg (\int_{-\tau_{ii}^k(t)}^{u-t} d_s\xi_{ii}^k(t,s)\bigg ) du +u_i\sum_{j=1}^n\|L_{ij}(t)\| \\
&\le \sum_{k=1}^p |d_{ii}^k(t)|\int_{t-\tau_{ii}^k(t)}^t |x_i'(u)|\, du +u_i\sum_{j=1}^n\|L_{ij}(t)\| .
\end{split}
\end{equation}
Using  \eqref{4.2}, we get
  \begin{equation}\label{ineq_2}
\int_{t-\tau_{ii}^k(t)}^t |x_i'(u)|\, du\le  \int_{t-\tau_{ii}^k(t)}^t \bigg (\sum_{l=1}^p |d_{ii}^l(u)|\int_{-\tau_{ii}^l(u)}^0 |x_i(u+s)|\, d_s\xi_{ii}^l(u,s)+\sum_{j=1}^n|L_{ij}(u)x_{j,u}|\bigg )\, du.  
\end{equation}
Consider any $k\in \{1,\dots,p\},j\in \{1,\dots,n\}$. For $t\in J$ and $u\in [t-\tau_{ii}^k(t),t]\subset [t-r,t]$, we now show   that
 \begin{equation}\label{claim1}\|x_{j,u}\|\le u_i.
 \end{equation}
  We separate the cases $u\ge t_0$ and $u<t_0$.

If $u\in [t_0,t]$, for $s\le 0$  we obtain $\frac{|x_j(u+s)|}{g(s)}\le |x_j(u+s)|\le u_i$ if $u+s\in [t_0,t]$, and 
$\frac{|x_j(u+s)|}{g(s)}=\frac{|x_j(t_0+s_1)|}{g(s)}\le \frac{|x_j(t_0+s_1)|}{g(s_1)}\le \|x_{j,t_0}\|_g< u_i$ if $u+s\le t_0$, where $s_1=u+s-t_0$ (note that $s\le s_1\le 0$).

If $u\in [t-\tau_{ii}^k(t),t_0]$, for $s\le 0$ define $s_1=u+s-t_0$, $s_2=u+s-(t_0-r)$.
If $t_0-r\le u+s$,  then $-r\le s_1\le 0$ and $\frac{|x_j(u+s)|}{g(s)}=\frac{|x_j(t_0+s_1)|}{g(s_1)}\frac{g(s_1)}{g(s)}\le m \frac{|x_j(t_0+s_1)|}{g(s_1)}\le m\|x_{j,t_0}\|_g< u_i$.
If $ u+s<t_0-r$ since $s\le s_2\le 0$, we have $\frac{|x_j(u+s)|}{g(s)}=\frac{|x_j(t_0-r+s_2)|}{g(s)}\le \frac{|x_j(t_0-r+s_2)|}{g(s_2)}\le \|x_{j,{t_0-r}}\|_g$. Since  
 $$\|x_{t_0-r}\|_g=\sup_{s\le 0} \frac{|x(t_0-r+s)|}{g(s)}\le m \|x_{t_0}\|,\q t\ge T_1,$$
 again we conclude that  $\frac{|x_j(u+s)|}{g(s)}<u_i$. This proves  \eqref{claim1}.

The above estimates also show that, for all $i,k$ and $t\in J, u\in [t-\tau_{ii}^k(t),t], s\le 0$, we have
$|x_i(u+s)|\le g(s)u_i.$
 Thus, using again \eqref{xi_normal} and the fact that $g$ is nonincreasing, we obtain 
\begin{equation}\label{claim2}\int_{-\tau_{ii}^l(u)}^0 |x_i(u+s)|\, d_s\xi_{ii}^l(u,s)\le g(-\tau_{ii}^l(u))u_i.
 \end{equation}
Inserting \eqref{claim1}, \eqref{claim2} in \eqref{ineq_2} yields
 \begin{equation}\label{ineq_3}
\int_{t-\tau_{ii}^k(t)}^t |x_i'(u)|\, du\le u_i \int_{t-\tau_{ii}^k(t)}^t \bigg (\sum_{l=1}^p |d_{ii}^l(u)g(-\tau_{ii}^l(u))+\sum_{j=1}^n\|L_{ij}(u)\|\bigg )\, du.  \\
\end{equation}
The above inequalities \eqref{ineq_1}, \eqref{ineq_3} and (H2) lead to
$$x_i'(t)+d_i(t)x_i(t) \le u_i\bigg (c_i^\tau (t)+\sum_{j=1}^n\|L_{ij}(t)\|\bigg )\le u_i d_i(t),\q t\in J,$$
 thus
 $x_i(t)\le x_i(t_0)\e^{-\int_0^t d_i(s)\, ds}+u_i(1-\e^{-\int_0^t d_i(s)\, ds})$ for $ t\in J.$
 For $t=t_1$, we derive $x_i(t_0)-u_i\ge 0$, which contradicts the assumption $u_i>\ell$. Hence,
  \eqref{claim} holds and the proof is complete.
\end{proof}

 In a similar way, the arguments presented in  Lemma \ref{lem3.2} can be pursued  for systems  \eqref{4.1}, as follows.
Let $e:\R\to\R^+$  be a  measurable, locally integrable function with $\int_0^\infty e(t)\, dt=\infty$, for which the operators given by \eqref{tildeL}
 are well-defined, $i,j=1,\dots,n$, and
denote $E(t)=\int_0^te(u)\, du$. By
 the change of variables $y(t)=\e^{E(t)}x(t)$,
the linear DDE \eqref{4.1}  is transformed into
\begin{equation}\label{y(t)2} y_i'(t)=e(t) y_i(t)-\sum_{k=1}^p d_{ii}^k(t) \tilde \ell_i^k(t)(y_{i,t})
+\sum_{j} \tilde L_{ij}(t)(y_{j,t}),\  i=1,\dots,n,\ t\ge 0,
\end{equation}
where
$\tilde \ell_i^k(t)(x_{i,t})=\int_{-\tau_{ii}^k(t)}^0 \e^{\int_{t+s}^t e(u)\, du}x_i(t+s)\, d_s\xi_{ii}^k(t,s)$ and $\tilde L_{ij}(t)$ are as  in \eqref{tildeL}. For $k=1,\dots,p$, define $\tilde d_{ii}^k(t)=d_{ii}^k(t)\int_{-\tau_{ii}^k(t)}^0 \e^{\int_{t+s}^t e(u)\, du}\, d_s\xi_{ii}^k(t,s)$, so that $ d_{ii}^k(t)\tilde \ell_i^k(t)(1)=\tilde d_{ii}^k(t)$; note that $|\tilde d_{ii}^k(t)|\le \e^{\int_{t-\tau_{ii}^k(t)}^t e(u)\, du} |d_{ii}^k(t)|$. For this system, we consider 
\begin{equation}\label{M(t)2}
\tilde M(t)=\tilde D(t)-\tilde C^\tau(t)-\tilde A(t),\q t\ge 0,
\end{equation}
with the matrices $\tilde D(t),\tilde C^\tau (t), \tilde A(t)$ defined according to the notation in \eqref{4.3}. We have $\tilde D(t)=\sum_{k=0}^p \tilde D_k(t)$ with  $\tilde D_0(t)=- e(t)I$ for $I$  the $n\times n$  identity matrix and  $\tau_{ii}^0(t)\equiv 0$ and $\tilde D_k(t)=\diag\, (\tilde d_{11}^k(t),\dots,\tilde d_{nn}^k(t))$ for $k=1,\dots,p$, $\tilde A(t)=\Big [\|\tilde L_{ij}(t)\|\Big ]$ and $\tilde {C^\tau} (t)=\diag (\tilde {c_1^\tau} (t),\dots, \tilde {c_n^\tau} (t)) $ with $\tilde {c_i^\tau} (t)\le  {\dbtilde {c_i^\tau}} (t)$, for
$$   {\dbtilde {c_i^\tau}} (t):=\sum_{k=1}^p \e^{\int_{t-\tau_{ii}^k(t)}^t e(u)\, du} |d_{ii}^k(t)|\int_{t-\tau_{ii}^k(t)}^t \bigg (e(s)+
\sum_{l=1}^p   \e^{\int_{t-\tau_{ii}^l(t)}^t e(u)\, du}|d_{ii}^l(s)|g(-\tau_{ii}^l(s))+\sum_{j=1}^n \|\tilde L_{ij}(s)\|\bigg )ds.
 $$
  If the transformed matrix $\tilde M(t)$ satisfies (H2), from Lemma \ref{lem4.1} and reasoning along the lines of the proof of Lemma \ref{lem3.2}, we deduce that the original system  \eqref{4.1} is asymptotically stable. Moreover, if one can choose $e(t)=\de$ for some $\de>0$ and $\tilde M(t)$ still satisfies (H2), the stability is exponential.

In an analogous way,  Theorems \ref{thm3.1}  can now be adapted to the present setting.


\begin{thm}\label{thm4.1}  
Consider system  \eqref{4.1} in ${\cal C}$. Assume  (H1*), (H3)  and one of the following conditions:
 
(i) the functions $d_{ii}^k(t), \| L_{ij}(t)\|$ are all bounded on $\R^+$, $ i,j=1,\dots,n, k=1,\dots,p$, and the matrix $M(t)$ in \eqref{M(t)} satisfies   (H4);

(ii)  $\liminf_{t\to\infty}d_i(t)>0$ for $i=1,\dots,n$, and there exist $\al>1,T\ge 0$ and a vector $v>0$ such that
$D(t)v\ge \al (C^{\tau} (t)+A(t))v$ for $t\ge T$.

  Then \eqref{4.1} is exponentially asymptotically stable.  \end{thm}

   \begin{proof}  
(i) Take $v={\bf 1}$ in (H4), so that 
  there exists $m>0$ such that $d_i(t)-c_i^\tau(t)-\sum_{j}\| L_{ij}(t)\|\ge m$, for $ i=1,\dots,n$ and $t\ge T$.  Since all the coefficients  are bounded, there exists $\eta>0$ such that, for $t$ sufficiently large, 
\begin{equation}\label{4.4_0} 
d_i(t)-\eta-(1+\eta)^2\Big [c_i^\tau(t)+\sum_{j}\|L_{ij}(t)\|\Big ]>0,\q i=1,\dots,n.
 \end{equation}
 Fix  $\eta>0$ as above. By Lemma \ref{lem3.3}, choose $\de\in (0,\eta)$ such that \eqref{3.8} holds.
From the  computations above, after the change of variables $y(t)=\e^{\de t}x(t)$, for $\tilde M(t)$ in \eqref{M(t)2} we have
 $$\tilde {c_i^\tau} (t)\le  {\dbtilde {c_i^\tau}} (t):=\sum_{k=1}^p \e^{\de \tau_{ii}^k(t)} |d_{ii}^k(t)|\int_{t-\tau_{ii}^k(t)}^t \left (\de+\sum_{l=1}^p 
 \e^{\de \tau_{ii}^l(s)} |d_{ii}^l(s)|g(-\tau_{ii}^l(s))+(1+\eta)\sum_{j=1}^n \|L_{ij}(s)\|\right )ds.$$
 Since  $\tau_{ii}^k(t)\le r$,  we may suppose that $\de>0$ is chosen so that 
 $\de \Big (1+r\e^{\de r}\sum_{k=1}^p |d_{ii}^k(t)|\Big )<\eta$ and $ \e^{\de \tau_{ii}^k(t)}\le \e^{\de r}<1+\eta$ for all $i,k$ and $t\ge 0$. We deduce that $ \tilde {\tilde {c_i^\tau}}  (t)\le (1+\eta)^2 c_i^\tau(t)+\de r\e^{\de r}\sum_{k=1}^p |d_{ii}^k(t)|$. From \eqref{4.4_0}, it follows that
 $(\tilde M(t){\bf 1})_i=d_i(t)-\de -\tilde {c_i^\tau} (t)-\sum_{j=1}^n \|\tilde L_{ij}(t)\|>0$ for $ t\ge T,\, i=1,\dots,n,$
so the matrix $\tilde M(t)$  satisfies (H2). The conclusion follows.

(ii) As above, take $v={\bf 1}$ in the inequality $D(t)v\ge \al (C^{\tau} (t)+A(t))v$, and consider $m>0$ such that 
$d_i(t)\ge m$ for $t\ge T$ and $i=1,\dots,n$. Take $\eta>0$ with $1+\eta<\al$, where $\al>1$ is such that
 \begin{equation}\label{4.4} 
d_i(t)\ge \al\Big [c_i^\tau(t)+\sum_{j}\|L_{ij}(t)\|\Big ],\q i=1,\dots,n,
 \end{equation}
and by Lemma \ref{lem3.3}, choose $\de\in (0,\eta)$ such that \eqref{3.8} holds and
 $$\de \e^{-\de r} \le \eta m,\q \de \le m[1-(1+\eta^3)\al^{-1}],\q \e^{\de \tau_{ii}^k(t)}\le \e^{\de r}<1+\eta.$$
 Note also that $\sum_{k=1}^p  |d_{ii}^k(t)|g(-\tau_{ii}^l(s))\ge \sum_{k=1}^p  |d_{ii}^k(t)|\ge d_i(t)\ge m$ for $t\ge T$.
 For ${\dbtilde {c_i^\tau}} (t)$ as above, after the change of variables $y(t)=\e^{\de t}x(t)$, for $\tilde M(t)$ in \eqref{M(t)2} we have $ \tilde {c_i^\tau} (t)\le {\dbtilde {c_i^\tau}} (t)$ with
 \begin{equation*}
 \begin{split}
  {\dbtilde {c_i^\tau}} (t):
 &\le \e^{\de r}\sum_{k=1}^p |d_{ii}^k(t)|\int_{t-\tau_{ii}^k(t)}^t \left (\de + \e^{\de r}\sum_{l=1}^p 
|d_{ii}^l(s)|g(-\tau_{ii}^l(s))+(1+\eta)\sum_{j=1}^n \|L_{ij}(s)\|\right )ds\\
&\le \e^{\de r}\sum_{k=1}^p |d_{ii}^k(t)|\int_{t-\tau_{ii}^k(t)}^t \left ((1+\eta) \e^{\de r}\sum_{l=1}^p 
|d_{ii}^l(s)|g(-\tau_{ii}^l(s))+(1+\eta)\sum_{j=1}^n \|L_{ij}(s)\|\right )ds\\
 &\le (1+\eta)^3 c_i^\tau(t).
 \end{split}
 \end{equation*}
  From \eqref{4.4} and the above choice of $\de$, for $ t\ge T,\, i=1,\dots,n$ we deduce that
 \begin{equation*}
 \begin{split}(\tilde M(t){\bf 1})_i&=d_i(t)-\de -\tilde {c_i^\tau} (t)-\sum_{j=1}^n \|\tilde L_{ij}(t)\|\\
 &\ge d_i(t)-\de -(1+\eta)^3 c_i^\tau (t) -(1+\eta) \sum_{j=1}^n \| L_{ij}(t)\|\\
 &\ge d_i(t)[1-(1+\eta)^3 \al^{-1}] -\de\ge 0,
  \end{split}
 \end{equation*}
 thus  the matrix $\tilde M(t)$  satisfies (H2).
\end{proof}

\begin{rmk}\label{rmk4.2}  {\rm Note that when all the functions  $d_{ii}^k(t), \| L_{ij}(t)\|,\tau_{ii}^k (t)$ are bounded on $\R^+$ and $M(t)$ in \eqref{M(t)} satisfies (H4), then $D(t)v\ge \al (C^{\tau} (t)+A(t))v$ for $t\ge T$  (cf.~(H5) and Remark \ref{rmk3.2}).
}\end{rmk}

\begin{cor}\label{cor4.1}  Consider   the scalar linear equation
\begin{equation}\label{4.6}
x'(t)=-\sum_{k=1}^p d_k(t)\int_{-\tau_k(t)}^0x(t+s)\, d_s\xi_k(t,s)+L_0(t) x_t,\q t\ge 0,
\end{equation}
where  $ \tau_k(t),d_k(t)$ are continuous, $\tau_k(t)\ge 0$ and bounded on $\R^+$, $L_0(t)\var=\be(t) \int_{-\infty}^0\var (s)\, d_s\nu(t,s)$
for $ (t,\var)\in [0,\infty)\times  C_g^0(\R)$, with  $\be(t)=\|L_0(t)\|$,  $s\mapsto\nu(t,s)\in M_g((-\infty,0];\R)$ and  $\nu(t,s)$ satisfying (H3),
$s\mapsto \xi_k(t,s)$ are nondecreasing and such that $\xi_k(t,0)-\xi_k(t, -\tau_k(t))=1$,  and  $ \xi_k(t,s),\nu(t,s)$   are continuous in $t$,   for $t\in\R^+,1\le k\le p$.  In addition, suppose that one of the following conditions holds:

(i) $ d_k(t),\be(t)$ are bounded, on $\R^+,1\le k\le p$, and    there exist $T>0,\vare>0$  such that
\begin{equation}\label{scalarH5}\sum_{k=1}^p d_k(t)\ge\vare+
\sum_{k=1}^p |d_k(t)|\int_{t-\tau_k(t)}^t   \left ( \sum_{l=1}^p |d_l(s)|g(-\tau_l(s))+\be(s)\right )ds+\be(t),\q t\ge T;
\end{equation}

(ii) $\liminf_{t\to\infty} \sum_{k=1}^p d_k(t)>0$ and there are $\al >1$ and $T\ge 0$ such that
\begin{equation}\label{scalarH5_2}\sum_{k=1}^p d_k(t)\ge\al\left [
\sum_{k=1}^p |d_k(t)|\int_{t-\tau_k(t)}^t   \left ( \sum_{l=1}^p |d_l(s)|g(-\tau_l(s))+\be(s)\right )ds+\be(t)\right],\q t\ge T.\end{equation}
Then \eqref{4.6} is exponentially asymptotically stable.   
\end{cor}


The case of linear DDEs  with only bounded delays, either  discrete or distributed,  is now addressed.

\begin{cor}\label{cor4.2}   Consider  the linear DDE in $C([-r,0];\R^n)$  
\begin{equation}\label{4.5}
x_i'(t)=-\sum_{j=1}^n\sum_{k=1}^p d_{ij}^k(t)\int_{-\tau_{ij}^k(t)}^0 x_j(t+s)\, d_s\xi_{ij}^k(t,s),\q t\ge 0, i=1,\dots,n,
\end{equation}
and assume that:

(h1)  $d_{ij}^k:\R^+\to\R,\tau_{ij}^k:\R^+\to\R^+$ are continuous,   and $\xi_{ij}^k(t,s)$ are  continuous on $t\in \R^+$,  of bounded variation on $s\in [-\tau_{ij}^k(t),0]$, with ${\rm Var}_{s\in [-\tau_{ij}^k(t),0]} \xi_{ij}^k(t,s)=1$,
for $i,j=1,\dots,n, k=1,\dots,p$;

(h2)
$s\mapsto \xi_{ii}^k(t,s)$ is nondecreasing for $ t\ge 0, 1\le i\le n, 1\le k\le p;$

 (h3) $d_i(t):=\sum_{k=1}^pd_{ii}^k(t)>0$  for $t\ge T_0$ and all $i,k$, for some $T_0\ge 0$;
 
 (h4) $\tau_{ij}^k(t)\in [0,r]$ for $ t\ge 0, 1\le i,j\le n, 1\le k\le p$;
 
 (h5)  all functions $d_{ij}^k$ are bounded on $\R^+$.\\
Define the matrices 
\begin{equation}\label{M(t)3}
\widehat D(t)=\big [ \widehat d_{ij}(t)\big ],\q
 C^\tau(t)=\diag\, (c_1^\tau (t),\dots, c_n^\tau(t)),
  \end{equation}  where
$$  \widehat d_{ij}(t) =\system{&-\sum_{k=1}^p|d_{ij}^k(t)|,&\ i\ne j\cr  &\q d_i(t),&\ i=j\cr},\  %
c_i^\tau (t)
= \sum_{k=1}^p |d_{ii}^k(t)|\int_{t-\tau_{ii}^k(t)}^t \sum_{j=1}^n\sum_{l=1}^p  |d_{ij}^l(u)|\, du\ (1\le i\le n).$$
If  $M(t):= \widehat D(t)- C^\tau(t)$
  satisfies (H4), then \eqref{4.5} is exponentially asymptotically stable.  
  In parti\-cular, under the above conditions, the system
\begin{equation}\label{Ber}
x_i'(t)=-\sum_{j=1}^n\sum_{k=1}^p d_{ij}^k(t)x_j(t-\tau_{ij}^k(t)),\q t\ge 0, i=1,\dots,n,
\end{equation}
is exponentially asymptotically  stable.
\end{cor}



\begin{rmk} {\rm   In a recent paper, with a technique  which makes uses of Bohl-Perron theorem and matrix norms,  Berezansky et al. \cite{BDSS18} gave sufficient conditions for the exponential asymptotic stability for linear systems  of the form
\eqref{Ber}, with  {\it bounded  discrete} delays and all coefficients {\it essentially bounded}  (although the more general framework of measurable, locally integrable functions  was considered).  However, in \cite{BDSS18} not only the analysis   is restricted to the case of discrete delays, but the hypotheses are stronger then the requirements in Corollary \ref{cor4.2}: for \eqref{Ber}, in \cite[Theorem 5]{BDSS18}  it is assumed  that:  (i)  $d_i(t)$ are bounded away below from 0,  (ii) there exists an autonomous matrix $N$ such that the matrix $N(t):=\Big [m_{ij}(t)/d_{i}(t)\Big ]$, where $M(t)=\Big [m_{ij}(t)\Big ]=\widehat D(t)-C^\tau(t)$ for the  matrices in \eqref{M(t)3}, satisfies $N(t)\ge N$ for $t\ge 0$, and (iii) $I-N$ is a non-singular M-matrix. We emphasize that this latter condition is equivalent to saying that $v-Nv\ge u$ for some positive vectors $u,v$, which is more restrictive than saying that $M(t)$ satisfies (H4). The approach in \cite{BDSS18} was extended  most recently by the same authors in  \cite{BDSS19}, where  new criteria for the exponential asymptotic stability  of system  \eqref{Ber} depending on all the delays $\tau_{ij}(t)$ were given.
An interesting open question is how to generalize the results in  \cite{BDSS19} to DDEs with distributed, and possibly unbounded, delays.
}\end{rmk}

Several  versions of Theorem \ref{thm3.4} can be stated  for the present framework.   To avoid repetitions and keep this manuscript in a reasonable size, in the formulation  below we assume that all the diagonal delays are bounded and all diagonal coefficients $d_{ii}^k(t)$ are nonnegative.

\begin{thm}\label{thm4.2}  For   \eqref{4.5}, and with the notation in the above corollary, assume (h1),(h2) and

(h3') all functions $d_{ii}^k$ are nonnegative on $\R^+$, with $d_i(t):=\sum_{k=1}^p d_{ii}^k(t)>0$;

(h4') $\tau_{ii}^k(t)\in [0,r]$ for  $ t\ge 0, 1\le i\le n, 1\le k\le p$.\\
Let 
$$M(t)= D(t)- \widehat C^\tau(t)- \widehat A^\tau (t)$$
  where
  $D(t)=\diag (d_1(t),\dots,d_n(t)), \widehat A^\tau (t)=\big [\widehat a_{ij}^\tau(t)\big ],
 \widehat C^\tau(t)=\diag\,( \widehat c_1^\tau(t),\dots, \widehat c_n^\tau(t)),$
 with,  for $1\le i,j\le n$ and $1\le k\le p$,
\begin{equation*}
\begin{split}
\widehat a_{ij}^\tau (t) &=\system{&\sum_{k=1}^p g(-\tau_{ij}^k(s))|d_{ij}^k(t)|,&\ i\ne j\cr  &\q 0,&\ i=j\cr}, %
\widehat c_i^\tau (t)
= \sum_{k=1}^p d_{ii}^k(t)\int_{t-\tau_{ii}^k(t)}^t \sum_{l=1}^p \sum_{j=1}^ng(-\tau_{ij}^l(u))|d_{ij}^l(u)|\, du.
\end{split}
\end{equation*} 
Assume  that there exist $\al>1,T\ge 0$ and a vector $v>0$  and a  measurable, locally integrable function $e:\R\to\R^+$  such that:

(i) $\min_{1\le i\le n}\left (M(t)]v\right)_i\ge e(t),t\ge T$;

(ii) $D(t)v\ge \al (\widehat C^{\tau} (t)+\widehat A^\tau(t)), t\ge T$;

(iii)   $\int_0^\infty e(t)\, dt=\infty$;

(iv) $\sup_{t\ge T}\int_{t-\tau(t)}^t e(u)\, du<\infty$,  for $\tau(t)=\max\{ \tau_{ij}^k(t): 1\le i,j\le n,1\le k\le p\}$.
\\
Then
\eqref{4.5} is asymptotically  stable. \end{thm}

\begin{proof} 
With the notations in   \eqref{4.3}, since  $L_{ii}(t)\equiv 0$,  $a_{ij}(t):=\|L_{ij}(t)\|\le \sum_{k=1}^p g(-\tau_{ij}^k(s))|d_{ij}^k(t)|= \widehat a_{ij}(t)$ for $i\ne j$ and
$$c_i^\tau(t)= \sum_{k=1}^p d_{ii}^k(t)\int_{t-\tau_{ii}^k(t)}^t \Big (\sum_l  g(-\tau_{ii}^l(u))d_{ii}^l(u) +\sum_{j\ne i} \|L_{ij}(u)\|\Big)\, du\le\widehat c_i^\tau(t)\q {\rm for}\q 1\le i\le n.$$
Thus we get $A(t)\le \widehat {A}^\tau (t)$ and  $C^\tau(t)\le \widehat C^\tau(t)$.
By the change  of variables $y(t)=\e^{\de E(t)}x(t)$ where $E(t)=\int_0^t e(u)\, du$ and $0<\de \ll 1$, system \eqref{4.5} is transformed into \eqref{y(t)2} with $e(t)$ replaced by $\de e(t)$.
 Arguing as in Theorem \ref{thm3.4}, one can show that the matrix $\tilde M(t)$ in \eqref{M(t)2} satisfies (H2), and
 the conclusion comes from Lemma \ref{lem4.1} applied to $y(t)$, since  solutions of \eqref{4.5} then satisfty $x(t)=\e^{-\de E(t)}y(t)\to 0$ as $t\to\infty$. Details are omitted.
\end{proof}


\section{Examples}
\setcounter{equation}{0}
We now illustrate our results with some  simple examples. The notation for $L_{ij}(t)$ in Sections 3 and 4 will be used.

\begin{exmp}\label{exmp2.1}  {\rm Consider a linear  system of the form
 \begin{equation}\label{2.19}
x_i'(t)=-d_it^2 x_i(t)+\sum_{j=1}^n\ga_{ij}(t)\int_{-\tau_{ij}(t)}^0 K_{ij}(s)x_j(t+s)\, ds,\q t\ge 0, i=1,\dots,n,\\
\end{equation}
where $d_i$ are positive constants, $\ga_{ij}(t), \tau_{ij}(t)$ are continuous and nonnegative, $K_{ij}$ are bounded and integrable on $[-\tau_{ij}(t),0]$, $1\le i,j\le n$. Write $\tilde K_{ij}(t,s)= K_{ij}(s)\chi_{[-\tau_{ij}(t),0]}(s)$ for $t\ge 0, s\le 0$.  Take ${\cal C}=C_g^0$ for any function $g$ satisfying the properties in (g) and such that, for some $\al>0$ and all $i,j$, $\sup_{t\in\R^+}\int_{-\tau_{ij}(t)}^0\e^{-\al s}g(s)\, ds<\infty$. Thus, (H3) is satisfied with any $\al_0\in (0,\al)$.
  The linear operators $L_{ij}(t)$ have norm $a_{ij}(t)=\ga_{ij}(t)\|g\tilde K_{ij}(t,\cdot)\|_{L^1}=\ga_{ij}(t)\int_{-\tau_{ij}(t)}^0 g(s)|K_{ij}(s)|\, ds$. 
If $$\ga_{ij}(t)=o(t^2)\q {\rm as}\q t\to\infty,$$
then (H5) holds, and Theorem \ref{thm3.1}(ii) implies that \eqref{2.19} is   exponentially asymptotically stable.}
\end{exmp}


\begin{exmp}\label{exmp2.1}   {\rm Consider a linear  system of the form
 \begin{equation}\label{2.19'}
x_i'(t)=-d_i(t) x_i(t)+\sum_{j=1}^n \ga_{ij}(t) \int_{t/2}^t K_{ij}(s)x_j(s)\, ds,\q t\ge 0, i=1,\dots,n,\\
\end{equation}
where $d_i,\ga_{ij}:\R^+\to \R^+$ are continuous, with $ \ga_{ij}(t)$  bounded, $d_i(t)\ge \frac{1}{t}$, $K_{ij}$ are continuous, integrable  on $\R^+$  and  there is $\al>0$ such that $\int_0^\infty \e^{\al s}|K_{ij}(s)|\, ds<\infty,\ 1\le i,j\le n.$
This equation has the form \eqref{LinTau} with unbounded delays $\tau_{ij}(t)=t/2$ for all $i,j$. Consider the space ${\cal C}=C_\ga^0$, for some  $\ga\in (0,\al)$. 

With the previous notation in \eqref{3.7}, $M(t)=\diag\, (d_1(t),\dots,d_n(t))-\big [\|L_{ij}(t)\|\big]$, $L_{ij}(t)\phi= \ga_{ij}(t) \int_{-t/2}^0 K_{ij}(t+s)\phi(s)\, ds$ for $\phi\in C_\ga^0(\R),\ 1\le i,j\le n$.
For any $\de >0$ and $T=T(\de)\ge 0$ sufficiently large, 
 \begin{equation*}
 \begin{split}
 a_{ij}(t):=\|L_{ij}(t)\|&=\ga_{ij}(t)\e^{\ga t} \int_{t/2}^t \e^{-\ga u} |K_{ij}(u)|\, du\\
 & \le  \ga_{ij}(t)\e^{(\ga-\al) t/2} \int_{t/2}^t  \e^{\al u}|K_{ij}(u)|\, du \le \frac{\de}{t},\q t\ge T,1\le i,j\le n .
 \end{split}
\end{equation*}
This shows that one can choose 
 $c\in (0,1)$ and $T>0$ such that
$$cd_i(t)\ge \sum_{j=1}^n a_{ij}(t), \q e_i(t):=d_i(t)-\sum_{j=1}^n a_{ij}(t)\ge \frac{c}{t},\q t\ge T.$$
 With $e(t)= \frac{c}{t}$ for $t\ge T$, we have $\int^\infty e(t)\, dt=\infty$ and
$\int_{t/2}^t e(s)\, ds=c\log2$ for $t\ge T$. 
From Theorem \ref{thm3.4}, it follows that the zero solution of \eqref{2.19'} is globally atractive. 
%
%
} \end{exmp}

\begin{exmp}\label{exmp2.3}  {\rm Consider the linear  system
 \begin{equation}\label{2.20}
x_i'(t)=-\sum_{ j=1}^nd_{ij}t^\al x_j(t)+\sum_{j=1}^nb_{ij}t^\al\int_{-\tau_{ij}(t)}^0 x_j(t+s)\, ds,\q i=1,\dots,n,\\
\end{equation}
where $\al>0, b_{ij}, d_{ij}\in \R$ with $d_i:=d_{ii}>0$ for all $i$, and the delays $\tau_{ij}(t)$ are continuous with  $0\le \tau_{ij}(t)\le r_{ij}$ for some constants $r_{ij}>0$, $i,j=1,\dots,n$. 

  With the  notation in \eqref{3.7} we have $d_{ij}(t)=d_{ij}t^\al, a_{ij}(t)=|b_{ij}|\tau_{ij}(t)t^\al \le |b_{ij}|r_{ij} t^\al$. Define the $n\times n$ matrices $\widehat{D}=\diag\, (d_1,\dots,d_n)-\Big [(1-\de_{ij}) |d_{ij}|\Big ]$, $|B|=\Big [|b_{ij}|r_{ij}\Big ]$, where $\de_{ij}=1$ if $i=j$ and $\de_{ij}=0$ if $i\ne j$,  and assume that
$$N:=\widehat{D}- |B|$$
 is a non-singular M-matrix. This is equivalent to saying that there exists a positive vector $v$ such that $u:=Nv>0$, hence  (H5) is satisfied.  From Theorem \ref{thm3.1}(ii) we deduce that \eqref{2.20} is exponentially asymptotically stable. Note however that none of coefficients is uniformly bounded on $\R^+$.
}
\end{exmp}

\begin{exmp}\label{exmp2.4} {\rm 
Take a function $g$ satisfying (g). 
Clearly $\psi\equiv 1\in C_g^0(\R)$ and $\|1\|_g=1$. By the Hahn-Banach theorem, there exists a functional $T\in (C_g^0(\R))'$ such that $T(1)=\|T\|=1$.
 In ${\cal C}=C_g^0(\R^2)$, consider the planar system 
\begin{equation*}
\begin{split}
x_1'(t)&=-2x_1(t)+T(x_{1,t})+(-1)^nT(x_{2,t})\\
x_1'(t)&=-2x_2(t)+T(x_{1,t})+T(x_{2,t})
\end{split}
\end{equation*}
%
with $n=1,2$. With the notation in \eqref{MM}, we have 
 $M_0=\left[ \begin{array}{cc}-1&(-1)^n \\
1&-1\end{array}\right]$ and $M=\left[ \begin{array}{cc}1&-1 \\
-1&1\end{array}\right]$, thus (H2) is satisfied with $v=(1,1)$.  For $n=2$,  $\la=0$ is a root of  its characteristic equation, which is given by $\big(\la +2-T(\e^{\la\cdot})\big)^2-T(\e^{\la\cdot})^2=0$;
since  $(c,c)$  are equilibria for any $c$ constant, the system is not asymptotically stable.  For the case $n=1$, $\det M_0\ne 0$ and  from Theorem \ref{thm3.3} the system is asymptotically stable.
}
\end{exmp}

\begin{exmp}\label{exmp2.5} {\rm Consider the  linear planar system
 \begin{equation}\label{2.21}
 \begin{split}
x_1'(t)&=-(1+\cos ^2t)x_1(t-\tau_{11}(t))+c_1(1+\sin^2t)x_2(t-\tau_{12}(t))\\
x_2'(t)&=-(1+\sin ^2t)x_2(t-\tau_{22}(t))+c_2(1+\cos^2t)x_1(t-\tau_{21}(t))\\
\end{split}
\end{equation}
where $c_i\ne 0$ and the delays $\tau_{ij}(t)$ are continuous and nonnegative  on $\R^+$ (and possibly unbounded), $i,j=1,2$.

First, consider the case $\tau_{11}(t)=\tau_{22}(t)\equiv 0$. With $\ga_i=|c_i|,\, i=1,2$, and the  notation in \eqref{3.7}, we have $D(t)=\diag (1+\cos^2t,1+\sin^2t)$ and
 \begin{equation}\label{5.5}
 \begin{split}
M(t)= D(t)-A(t)&=\left [\begin{matrix}
1+\cos^2t&-\ga_1(1+\sin^2t)\\
-\ga_2(1+\cos^2t)&1+\sin^2t
\end{matrix}\right ].
\end{split}
\end{equation}
For a vector $v=(1,v_2)$ with $v_2>0$,  write $M(t)v=\left [\begin{matrix} e_1(t)\\ e_2(t)v_2\end{matrix}\right ]$. Since $\min e_1(t)=1-2v_2\ga_1, \min e_2(t)=1-2v_2^{-1}\ga_2$, if $4\ga_1\ga_2<1$, i.e., if
 \begin{equation}\label{gammas} 4|c_1c_2|<1,
 \end{equation}
  choosing $v_2$ such that $2\ga_2< v_2< (2\ga_1)^{-1}$, condition (H4) is satisfied with $v=(1,v_2)$. From Theorem \ref{thm3.1}, \eqref{2.21} is exponentially asymptotically stable. 

Secondly, let $\tau_{11}(t)>0,\tau_{22}(t)>0$ for some $t>0$, but assume that all the delays $\tau_{ij}(t)$ are uniformly bounded,  $\tau_{ij}(t)\le r_{ij}$ on $\R^+$ with $\max_{i,j=1,2}r_{ij}=r>0$,  so that we work on ${\cal C}=C([-r,0];\R^2)$ (and $g\equiv 1$ on $[-r,0]$). For $C^\tau (t)=\diag\, (c_1^\tau(t),c_2^\tau(t))$ defined by \eqref{4.3}, we have 
\begin{equation*}
 \begin{split}
c_1^\tau(t)=
(1+\cos^2t) \int_{t-\tau_{11}(t)}^t (1+\cos^2u+\ga_1(1+\sin^2u))\, du\\
c_2^\tau(t)=
(1+\sin^2t)\int_{t-\tau_{22}(t)}^t(\ga_2( 1+\cos^2u)+1+\sin^2u)\, du,
\end{split}
\end{equation*}
and rough estimates give
\begin{equation*}
 \begin{split}
c_1^\tau(t)\le 2 r_{11}
(1+\cos^2t) (1+\ga_1),\\
c_2^\sigma(t)\le 2 r_{22}
(1+\sin^2t) (1+\ga_2).
\end{split}
\end{equation*}
For  $\widehat D(t),C^\tau (t)$ defined by \eqref{M(t)3}, the matrix $\widehat D(t)$ coincides with the matrix $D(t)-A(t)$  in \eqref{5.5}. Proceeding as above we deduce that it is possible to choose a  vector $v=(1,v_2)$ with $v_2>0$ such that $M(t)=\widehat D(t)-C^\tau (t)$ satisfies (H4), provided that
 \begin{equation}\label{gammas2} 4|c_1c_2|<\big (1-2r_{11}(1+|c_1|)\big)\big (1-2r_{22}(1+|c_2|)\big),
 \end{equation}
and from Corollary \ref{cor4.2} we conclude that \eqref{2.21} is exponentially asymptotically stable. 
Note that when $r_{ii}= 0\, (i=1,2)$, condition \eqref{gammas2} reduces to \eqref{gammas}.
}
\end{exmp}

%

\begin{exmp}\label{exmp2.6} {\rm Consider the  following scalar equation:
\begin{equation}\label{5.8}
x'(t)=-(1+\cos^2t) x(t-\tau_1(t))+(1-\sin^2t)x(t-\tau_2(t))+b\int_{t/2}^t\e^{-\ga s}x(s)\, ds,\q t\ge 0,
\end{equation}
where  $b,\ga>0$, $\tau_i:\R^+\to\R^+$ are continuous with $\tau_i(t)\le r_i,\, i=1,2$. With the notations in Corollary \ref{cor4.1}, $d(t)=1+\cos^2t -(1-\sin^2t)=1, |d_1(t)|+|d_2(t)|=2+\cos^2 t-\sin^2t$ and $L_0(t)$ is 
the operator defined by $L_0(t)\phi=b\e^{-\ga t}\int_{-t/2}^0 \e^{-\ga s}\phi(s)\, ds$. Fix the phase space ${\cal C}=C_\ga^0(\R)$.
Thus, the  norm $\be(t):=\|L_0(t)\|$ is given by $\be(t)=b\e^{-\ga t}\int_{-t/2}^0 \e^{-2\ga s}\, ds=\frac{b}{2\ga}(1-\e^{-\ga t})$. With $r=\max(r_1,r_2)$, the coefficient $c^\tau(t)$  in \eqref{4.3} satisfies
\begin{equation*}
\begin{split}c^\tau(t)&\le (2+\cos^2 t-\sin^2t)
\int_{t-r}^t \left (\e^{\ga r}  (2+\cos^2 u-\sin^2u)+ \frac{b}{2\ga}(1-\e^{-\ga u})\right)\, du\\
&\le 3r\Big [3\e^{\ga r}+\frac{b}{2\ga}\Big ]+O(\e^{-\ga t}),\q {\rm as}\q t\to \infty.
\end{split}
\end{equation*}
Thus, if  $9\e^{\ga r}r+\frac{3b}{2\ga}(1+r)<1$, then for any $\vare >0$ there is $T>0$ such that
$$c^\tau(t)+\be(t)\le 3\Big [3\e^{\ga r}r+\frac{b}{2\ga}(1+r)\Big ]+\vare<1,\q t\ge T.
$$ In particular, \eqref{5.8} is exponentially asymptotically stable
provided that $3b<2\ga$ and $r>0$ is sufficiently small. }\end{exmp}

\section*{Acknowledgements}
This work was  supported by National Funding from FCT - Funda\c c\~ao para a Ci\^encia e a Tecnologia (Portugal) under project UIDB/04561/2020.

\end{document}